\documentclass[a4paper, 11pt]{article}
\usepackage[cp1251]{inputenc}
\usepackage[english]{babel}
\usepackage{amsthm}
\usepackage{cite}
\usepackage{tikz-cd}
\usepackage{amsmath}
\usepackage{amsfonts}
\usepackage{amssymb}
\usepackage{hyperref}
\usepackage[]{authblk}
\usepackage{etaremune}
\usepackage{enumitem}

\DeclareMathOperator{\Spec}{Spec} \DeclareMathOperator{\Aut}{Aut}
\DeclareMathOperator{\End}{End} \DeclareMathOperator{\Char}{char}
 
 \DeclareMathOperator{\Ker}{Ker}
\DeclareMathOperator{\Mat}{M} 
 
\DeclareMathOperator{\Deg}{deg}

\DeclareMathOperator{\Id}{Id} 
 
 \DeclareMathOperator{\GL}{GL}
\DeclareMathOperator{\diag}{diag} 
\DeclareMathOperator{\gr}{gr} \DeclareMathOperator{\Ind}{Ind}
\newcommand{\TAut}{\operatorname{TAut}}

\newtheorem{thm}{Theorem}[section]
\newtheorem{lem}[thm]{Lemma}

\newtheorem{prop}[thm]{Proposition}
\newtheorem{Cprop}[thm]{Conjectural Proposition}
\newtheorem{cor}[thm]{Corollary}
\newtheorem{conj}[thm]{Conjecture}
\newtheorem{Def}[thm]{Definition}
\newtheorem{remark}[thm]{Remark}
\newtheorem{example}[thm]{Example}

\begin{document}

\title{\bf Torus actions on free associative algebras, lifting and Bia\l{}ynicki-Birula type theorems}
\renewcommand\Affilfont{\itshape\small}
\author[2]{Alexei Belov-Kanel\thanks{kanel@mccme.ru}}
\author[3]{Andrey Elishev\thanks{andrey.elishev@gmail.com}}
\author[4,7]{Farrokh Razavinia\thanks{f.razavinia@phystech.edu}}
\author[2]{Louis Rowen\thanks{rowen@math.biu.ac.il}}
\author[5]{Jie-Tai Yu\thanks{jietaiyu@szu.edu.cn,
jietaiyu@163.com}}
\author[1,6]{Wenchao Zhang\thanks{Correspondence: zhangwc@hzu.edu.cn}}

%\thanks{The research conducted in this work was in part supported by a grant from IPM (No. 1403170014)}

\affil[1]{School of Mathematics and Statistics, Huizhou University, Huizhou 516007, China}
\affil[2]{Department of Mathematics, Bar-Ilan University, Ramat-Gan 52900, Israel}
\affil[3]{Moscow Institute of Physics and Technology, Dolgoprudny, Moscow Region 141700, Russia}
%\affil[4]{Faculty of Sciences, Mathematics Group, Azarbaijan Shahid Madani University, Tabriz, Iran} ask him.
\affil[4]{Department of Mathematics, Azarbaijan Shahid Madani University, Tabriz 5375171379, Iran}
\affil[5]{College of Mathematics and Statistics, Shenzhen University, Shenzhen 518061, China}
\affil[6]{Department of Mathematics, Southern University of Science and Technology, Shenzhen 518055, China}
\affil[7]{School of Mathematics, Institute for Research in Fundamental Sciences (IPM), P.O. Box 19395-5746, Tehran, Iran.}

%\renewcommand{\thefootnote}{\fnsymbol{footnote}}
%\footnotetext{{\it 2010 Mathematics Subject Classification:} 16S10, 14R10, 14R20}
%\footnotetext{{\it Key words:} Torus action, linearity, free algebra, automorphism lifting.}
%\renewcommand{\thefootnote}{\arabic{footnote}}
%\fontsize{12}{12pt}\selectfont

%\date{}

\maketitle

\renewcommand{\abstractname}{Abstract}

\begin{abstract}
We examine the problem of the linearity of an algebraic torus action
in the associative setting. We prove the free algebra analog of a
classical theorem of Bia\l{}ynicki-Birula, which establishes
linearity of maximal torus action. Additionally, we formulate and prove linearity theorems for specific classes of regular actions, and provide a framework for constructing non-linearizable actions, analogous to the work of Asanuma. This framework has applications in the study of the Associative Cancellation Conjecture. Furthermore, we show the existence of two non-isomorphic algebras, whose free products with a polynomial ring are isomorphic.
\end{abstract}

\qquad \qquad Dedicated to the memory of Professor Nikolai Alexandrovich Vavilov.

\medskip

\medskip

{\it 2020 Mathematics Subject Classification:} 16S10, 14R10, 14R20

{\it Key words:} Torus action, linearity, free algebra, automorphism lifting.

\section{Introduction}

%In this note we consider algebraic torus actions on the affine space, according to Bia\l{}ynicki-Birula, and formulate certain noncommutative generalizations.

\subsection{The linearity of torus action problem}

The \textbf{group action linearity problem} asks, generally speaking, whether any action of a given algebraic group on an affine
space is \textbf{linearizable}, i.e., linear in a suitable coordinate system (or, in other words, whether for any such action
there exists an automorphism of the affine space that conjugates the action to a representation). This subject owes its existence largely to the classical work of A. Bia\l{}ynicki-Birula \cite{BB1}, who considered \emph{regular} actions of the
$n$-dimensional torus on the affine space $\mathbb{A}^n$  (over an algebraically closed ground field) and proved that any effective
action is linearizable. Bia\l{}ynicki-Birula's results motivated the study of various analogous instances, such as those that are dealing with the actions of a tori of dimension smaller than that of the affine space, or, alternately, linearity conjectures that arise when the torus is replaced by a different algebraic group. In particular,
Bia\l{}ynicki-Birula himself \cite{BB2} proved that any effective action of an $(n-1)$-dimensional torus on $\mathbb{A}^n$ is
linearizable, and for a while it was believed \cite{KR} that the same was also true for a torus and an affine space of an arbitrary dimension over any field. However, this general linearity conjecture was subsequently disproved by Asanuma \cite{Asanuma}.

\smallskip

Defining a regular group action on an affine space is essentially (i.e. modulo duality) the same as defining an action on its affine coordinate algebra, which is given by $\mathbb{K}[x_1,\ldots,x_n]$. Various
instances of the group action linearity problem therefore admit a purely algebraic reformulation.

\smallskip

With that observation in mind, it becomes natural to attempt to generalize the group action linearity problem to the category of associative algebras over a field. This involves replacing the geometric notions with their algebraic counterparts, and then further replacing them with the associative algebra analogues, in accordance with the commonly accepted wisdom of noncommutative geometry.

\smallskip

The most obvious setting to look for in that regard is provided by the \emph{free associative} $\mathbb{K}$-algebra $F_n(\mathbb{K})$ in $n$ generators ($\mathbb{K}$ is a field; throughout the paper, we will frequently denote $F_n(\mathbb{K})$ also by $F_n$ while also referring to it as the free algebra in contexts involving associative $\mathbb{K}$-algebras). The free algebra $F_n$ in $n$ generators $z_1,\ldots, z_n$ is the $\mathbb{K}$-module
$$
F_n = \mathbb{K}\langle z_1,\ldots, z_n\rangle
$$
consisting of words $z_{i_1}z_{i_2}\ldots z_{i_s}$ in the alphabet $\{z_1,\ldots, z_n\}$, with the algebra multiplication given by the $\mathbb{K}$-linearly extended word concatenation. It is the free object in the category of associative unital $\mathbb{K}$-algebras and can thus be regarded as the associative counterpart of the commutative polynomial algebra $\mathbb{K}[x_1,\ldots, x_n]$. And since the latter algebra serves as a centerpiece in the discussion involving algebraic torus actions on the affine spaces, the direct generalization of the various linearization results to the case of the free associative algebra becomes both natural and pertinent (for references pertaining to the latter point cf. Subsection 1.2 of the present paper). Note that, in fact the free algebra counterpart of the Bia\l{}ynicki-Birula's maximal torus linearization result \cite{BB1} was established in \cite{TA1}.

\smallskip

Given the existence of the free algebra version of the Bia\l{}ynicki-Birula theorem, one may inquire whether various other
instances of the linearity problem (such as the Bia\l{}ynicki-Birula theorem on the action of the $(n-1)$-dimensional torus on $\mathbb{K}[x_1,\ldots,x_n]$) can be generalized to the associative setting in a similar manner. As it turns out, direct adaptation of proof techniques from the commutative realm is
sometimes possible. However, there are certain limitations. For instance, Bia\l{}ynicki-Birula's proof in \cite{BB2} of the
linearity of $(n-1)$-dimensional torus actions uses commutativity in an essential way. Nevertheless, the exceptional nature of the $n=2$ case, specifically the tameness of
the polynomial algebra and the free associative algebra automorphisms in the planar case, is sufficient to provide the existence of results analogous to the original theorems of Bia\l{}ynicki-Birula, as we observe in Section 4 of the present work. We include these results for the sake of completeness of the picture.

\smallskip

The case of the positive-root torus actions on the free algebra is rather interesting, in the sense that its analysis has connections with the Jacobian Conjecture (as well as its free
associative version) and to other various techniques native to the general algebra. In particular, we discuss a certain nontrivial
approach to the problem of invertibility of endomorphisms of the free algebra originated by A.V. Yagzhev \cite{Yag1, Yag2, Yag3,
Yag4} (cf. also \cite{BBRY}). Yagzhev's work is deep and highly nontrivial; a detailed examination of the ideas presented there may
prove extremely beneficial, especially in light of its relation to the Jacobian Conjecture. It is also of great interest in connection with the groundbreaking work of Y.~Tsuchimoto \cite{Tsu1, Tsu2, Tsu3}.

\smallskip

Some of the proof techniques developed in the
classical work of Asanuma \cite[\S 2--5]{Asanuma} admit free associative analogues. This allows for the construction of a framework to prove the existence of non-linearizable torus actions in positive characteristic, analogous to Asanuma's work. Somewhat surprisingly, it also provides insights into the Associative Cancellation Conjecture.

\smallskip

Gupta \cite{Gupta1, Gupta2, Gupta3, Gupta4} has discovered positive-characteristic counterexamples to the Zariski Cancellation problem, building upon Asanuma's result \cite{Asanuma} on non-linearizable torus actions. Some of the intricacies of the Associative Cancellation Conjecture, particularly the possible existence of the positive-characteristic counterexamples and the way the latter are related to non-rectifiable epimorphisms and non-linearizable torus actions, might be elucidated by lifting Gupta's constructions to the free associative algebra!

\smallskip

\smallskip

The present paper serves as a record of our various works focused on the torus action linearity problem and related developments. This exposition largely follows a review pattern, particularly in the case of noncommutative (free algebra) generalizations of known results and techniques from the algebraic geometry, where the nature of these generalizations is noteworthy. The main results of the paper are as follows.

\smallskip

Firstly, we establish the free algebra version of the Bia\l{}ynicki-Birula theorem,
formulated and proved as in \cite{TA1}:
\begin{thm} %\label{BBfree}
Any effective regular (cf. Definition \ref{defactionregularfree}) action $\sigma$ of the algebraic $n$-torus $\mathbb{T}_n$ on the free algebra $F_n$  is linearizable.
\end{thm}

\smallskip

As a side note, which is a straightforward consequence of the classical results of Jung -- Van der Kulk, and Makar-Limanov on the tameness of automorphisms of the free algebra $F_2$ in two indeterminates, we mention the following:

\begin{prop}\label{propauxsection4}
Let $\sigma$ be an effective regular action of the algebraic $r$-torus $\mathbb{T}_r$ on the free algebra $F_2$ in two indeterminates such that the action on the (commutative) polynomial algebra given by the abelianization of $\sigma$ is linearizable. Then $\sigma$ itself is linearizable.
\end{prop}

%Given the existence of a free algebra version of the
%Bia\l{}ynicki-Birula theorem, one may inquire whether various other
%instances of the linearity problem (such as the Bia\l{}ynicki-Birula
%theorem on the action of the $(n-1)$-dimensional torus on
%$\mathbb{K}[x_1,\ldots,x_n]$) can be studied. As it turns out, sometimes,
%direct adaptation of the techniques used in the proof from the commutative realm is
%possible. However, there are certain limitations. For instance, in the Bia\l{}ynicki-Birula's proof \cite{BB2} of the linearity of $(n-1)$-dimensional torus actions, the commutativity has been used in
%an essential way. Nevertheless, as we show in Section \ref{sec4}, a neat bypass of that hurdle can be performed when $n=2$.

In a related development, we study the special class of regular \emph{positive-root torus actions} on the free algebra $F_n$ and obtain the positive answer to the linearity problem.

\begin{thm}\label{thmkrfree-over}
 Any effective regular torus action $\sigma:\mathbb{T}_r\times F_n\to F_n$  with positive roots is linearizable.
\end{thm}

Finally, some of the proof techniques developed in the classical work of Asanuma \cite[\S 2--5]{Asanuma} admit direct free
associative analogs that enable us to obtain sufficient conditions for the existence of \emph{non-linearizable torus actions} in positive characteristic (See Section \ref{sec6}), in complete analogy with Asanuma's results.

\begin{thm}\label{non-linear-exam}
 Let $\mathbb{K}$ be an arbitrary field. If for $n>m$ there exists a non-rectifiable epimorphism $\mathbb{K}\langle z_1,\dots,z_n\rangle\to \mathbb{K}\langle z_1,\dots,z_m\rangle$, then for $1\leq r\leq m+1$ there is a non-linearizable effective regular $\mathbb{T}_r$-action
  $$\phi: \mathbb{K}\langle u, y_1,\ldots, y_m, z_1, \ldots, z_n\rangle\to \mathbb{K}[t_1,\ldots, t_r, t_1^{-1},\ldots, t_r^{-1}]\langle u, y_1,\ldots, y_m, z_1, \ldots, z_n\rangle.$$
\end{thm}

As an interesting accompanying result to the techniques developed during the investigation of Theorem \ref{non-linear-exam}, we obtain a proposition concerning the free associative generalization of the Cancellation Conjecture, as formulated by Drensky and Yu \cite{DrYu}:

\begin{conj}[Associative Cancellation Conjecture, \cite{DrYu}]\label{conjdryu-overview}
Let $R$ be a $\mathbb{K}$-algebra. If
$$
R*\mathbb{K}\langle y\rangle\simeq_{\mathbb{K}} \mathbb{K}\langle
z_1,\ldots, z_{n-1}, y\rangle
$$
then
$$
R\simeq_{\mathbb{K}}\mathbb{K}\langle z_1,\ldots, z_{n-1}\rangle.
$$
\end{conj}

%In accordance with the general ideas presented in the present paper, extending the generalized cancellation conjecture \cite{BMLY, BY} to the associative category becomes an interesting problem. Gupta \cite{Gupta1, Gupta2, Gupta3, Gupta4} has discovered positive-characteristic counterexamples to the Zariski cancellation problem expanding upon Asanuma's result \cite{Asanuma} on non-linearizable torus actions.

\smallskip

The associative generalization of Asanuma's results on Rees algebras, as explored in Section \ref{sec6}, allows us to establish a version of the Associative Cancellation Conjecture for coproducts over a commutative $\mathbb{K}$-algebra $D$. % this place should be change to a pseduo Rees algebra

\begin{thm}\label{thmcancellation-overview}
Given a nonzero element $t\in \mathbb{K}[x]$ and polynomials $f(x)$ and $g(x)$ in $\mathbb{K}[x]$. Let $A=\mathbb{K}[x]\langle t,t^{-1}f(x)\rangle$ and $B=\mathbb{K}[x]\langle t,t^{-1}g(x)\rangle$ be the left pseudo Rees algebras of $(f(x))$ and $(g(x))$ respectively. 

\begin{enumerate}
    \item If $\mathbb{K}[x]/(f(x))\simeq_{\mathbb{K}} \mathbb{K}[x]/(g(x))$,
    then
    $$
    A * \mathbb{K}\langle y\rangle\simeq B * \mathbb{K}\langle y\rangle.
    $$
   \item If $f(x)\neq g(ax+b)$ for all $a,b\in \mathbb{K}$, then $A\not\cong B$.
\end{enumerate}
\end{thm}

Whence, we can establish the following result on the existence of non isomorphic algebras whose free products with a polynomial ring are isomorphic.
\begin{thm}
    There exists two non isomorphic $\mathbb{K}$-algebras $A$ and $B$,
    such that  $$
    A * \mathbb{K}\langle x\rangle \cong B *\mathbb{K}\langle x\rangle.
    $$
\end{thm}

\subsection{Some remarks on the related developments}

In recent years, the linearity of the effective torus actions has been proven to be a useful tool in the study of the geometry of
automorphism groups. In the papers \cite{KBYu} and \cite{KBEYu}, the following result was obtained.
\begin{remark}        \label{commautautthm}
 For $\mathbb{K}$ an algebraically closed field, and $n\geq 3$, every $\Ind$-scheme automorphism $\Phi$ of the $\Ind$-group
$\Aut(K[x_1,\dots,x_n])$ is inner.
\end{remark}

A generalization of this result to the automorphisms of the Cremona group in $d$ variables was given by Zariski homeomorphisms in \cite{UrZ}.

The notions of $\Ind$-variety (or $\Ind$-group in the present context) and $\Ind$-variety morphism were introduced by Shafarevich \cite{Shafarevich}. Note that an $\Ind$-variety is the direct limit of a system whose morphisms are closed embeddings, and the automorphism groups of algebras with polynomial identities, such as the (commutative) polynomial algebra, and of the free (associative) algebra, are archetypal examples, i.e. they are the corresponding direct systems of varieties consist of sets $\Aut^{\leq N}$ of automorphisms of total degree less or equal to a fixed number, with the degree induced by the grading. The morphisms are inclusion maps which are obviously closed embeddings.

Theorem \ref{commautautthm} is proved by means of tame approximation (stemming from the main result of \cite{An}), with the following Proposition, originally due to E. Rips (the article \cite{KBEYu} handles the subtleties of the positive-characteristic case), constituting one of the key results.
\begin{prop} \label{proprips}
Let $\mathbb{K}$ be algebraically closed and let $n\ge 3$. Suppose that $\Phi$ preserves the standard maximal torus action on the commutative polynomial algebra\footnote{That is, the action of the $n$-dimensional torus on the polynomial algebra $\mathbb{K} [x_1,\ldots,x_n]$, which is dual to the action on the affine space.}. Then $\Phi$ preserves all tame automorphisms.
\end{prop}

The proof relies on the Bia\l{}ynicki-Birula theorem on the maximal torus action. In a similar fashion, the paper \cite{KBYu} examines
the $\Ind$-group $\Aut \mathbb{K}\langle z_1,\ldots, z_n\rangle $ of automorphisms of the free algebra $\mathbb{K}\langle z_1,\ldots,
z_n\rangle$ in $n$ variables, and establishes results completely analogous to Theorem \ref{commautautthm} and Proposition
\ref{proprips}.\footnote{The free case is responsive to the above approach when $n>3$.} In accordance with that, the free associative
version of the Bia\l{}ynicki-Birula theorem was required.

\smallskip

Such an analog is indeed valid, and it was established  in the article \cite{TA1}. As stated, we will review the proof of this result in the present paper.

\subsection*{Acknowledgments}
\indent The authors thank the anonymous referees for their careful reading and valuable comments, which have improved the readability of this article.

The work is supported by the Russian Science Foundation (\href{https://rscf.ru/project/22-11-00177/}{No. 22-11-00177}).

L.R. is supported by the Israel Science Foundation (No. 1994/20).

W.-C. Z. would like to thank B.-C. Liu, and the Department of Mathematics of the Southern University of Science and Technology for their hospitality during his visit. W.-C. Z. is also supported by the GuangDong Basic and Applied Basic Research Foundation (Nos. 2022A1515110634, 2023A1515011690) and the Characteristic Innovation Project of Guangdong Provincial Department of Education (No. 2023KTSCX145).

F.R. is supported by the Azarbaijan Shahid Madani University under the grant contract No. 117.d.22844 - 08.07.2023. F.R. is also partially supported by a grant from the Institute for Research in Fundamental Sciences (IPM), with the grant (No. 1403170014). Also F.R. would like to express his warmest thanks and gratitude to the late Professor N. A. Vavilov, whom he sees himself indebted for the continuation of his scientific life, as being a member of his 2021 PhD defense committee.

A.E. is thankful to V.L. Dolnikov for fruitful discussions during the preliminary stages of this project and is especially grateful to the late Professor N.A. Vavilov for his many helpful remarks and attention to this study as well as for having been a member of A.E.'s 2019 PhD defense committee.

\section{Actions of algebraic tori}

In this section we recall basic definitions of the theory of torus actions, as formulated by Bia\l{}ynicki-Birula \cite{BB1, BB3} and others (cf. \cite{RefB1}).

\smallskip

Let $\mathbb{K}$ denote the ground field. Throughout, we assume $\mathbb{K}$ to be algebraically closed. There are no specific assumptions on the characteristic of $\mathbb{K}$ in Sections 3 -- 6.

Let $\mathbb{G}_m = \Spec \mathbb{K}[t, t^{-1}]$ be the algebraic group corresponding (in terms of the group of $\mathbb{K}$-points) to $\mathbb{K}^{\times}$ -- the multiplicative group of the ground field.

\begin{Def} \label{def-alg-torus}
An $r$-dimensional algebraic torus (or an $r$-torus) $\mathbb{T}_r$ is an algebraic group isomorphic (over $\mathbb{K}$) to the product $\mathbb{G}_m^r = \mathbb{G}_m\times\ldots\times\mathbb{G}_m$ of $r$ copies of the multiplicative group of $\mathbb{K}$.
\end{Def}

\smallskip

%It is usually denoted by $G_m$ and is the affine algebraic group  $\text{Spec}(\mathbb{K}[t,t^{-1}])$. An $n$-dimensional algebraic torus over $\mathbb{K}$ is an algebraic group $\mathbb{T}_n$ isomorphic to a finite direct product $\mathbb{K}^{\times}\times \ldots \times \mathbb{K}^{\times}$ which is a type of commutative affine algebraic group.

Denote by $\mathbb{A}^n$ the affine space of dimension $n$ over $\mathbb{K}$.

\begin{Def} \label{defaction}
A (left, geometric) $r$-torus action is a mapping
$$
\sigma: \mathbb{T}_r\times \mathbb{A}^n\rightarrow \mathbb{A}^n,
$$
such that
%for every $t\in \mathbb{T}_r$ the map $\sigma(t,-)$ is an automorphism of $\mathbb{A}^n$ and such that
the usual group action axioms (of identity and compatibility) are in place:
$$
\sigma(1,x)=x,\;\;\sigma(t_1,\sigma(t_2,x))=\sigma(t_1t_2,x).
$$
\end{Def}
The group action axioms imply that for every $t\in \mathbb{T}_r$ the map $\sigma(t, -)$ is a bijection of $\mathbb{A}^n$.

\begin{Def}[Effective action]\label{defactioneffective}
A torus action $\sigma$ is called effective if for every $t\neq 1$ there is an element $x\in \mathbb{A}^n$ such that $\sigma(t,x)\neq x$.
\end{Def}

The property of a torus action $\sigma$ that is most pertinent to the present paper (as well as, arguably, to the original works by Bia\l{}ynicki-Birula) is given by the following definition.
\begin{Def}[Regular action]\label{defactionregular}
A torus action $\sigma$ is regular if the mapping
$$
\sigma: \mathbb{T}_r\times \mathbb{A}^n\rightarrow \mathbb{A}^n.
$$
is a morphism of varieties.
\end{Def}

\smallskip

\begin{Def}\label{defcharacter}
Given an algebraic $r$-torus $\mathbb{T}_r$, the group of characters $\widehat{T}_r$ is the group of algebraic group morphisms
$$
\chi: \mathbb{T}_r\rightarrow \mathbb{G}_m.
$$
\end{Def}

It is well known that if $\mathbb{T}_r = (\mathbb{G}_m)^r$ and $t = (t_1,\ldots, t_r)\in \mathbb{T}_r$ then $\chi(t) = t_1^{\alpha_1}\ldots t_r^{\alpha_r}$ with $\alpha_i\in \mathbb{Z}$ integers -- that is, the group $\widehat{T}_r$ of characters of the algebraic torus $\mathbb{T}_r = (\mathbb{G}_m)^r$ is isomorphic to the (additive) group $\mathbb{Z}^r$.

\smallskip

Any regular $r$-torus action $\sigma$ corresponds to a homomorphism of $\mathbb{K}$-algebras
$$
\tilde{\sigma}: A\rightarrow A\otimes \mathbb{K}[t_1,\ldots, t_r, t_1^{-1},\ldots, t_r^{-1}]\simeq A[t_1,\ldots, t_r, t_1^{-1},\ldots, t_r^{-1}],
$$
where the algebra $A=\mathcal{O}(\mathbb{A}^n)$ of globally regular functions is identified with the polynomial algebra $\mathbb K[x_1,\ldots, x_n]$. For the sake of brevity, we denote this action by $\sigma$ as well whenever no confusion arises from the choice of notation.

For $t\in \mathbb{T}_r$ and $f\in A$, consider $\sigma(f) = \sum \chi(t)f_{\chi}$. It is well known \cite{RefB, KR} that the assignment $f\mapsto (\ldots, f_{\chi},\ldots)$ induces a direct sum decomposition
$$
A = \bigoplus_{\chi} A_{\chi}\;\;\text{where}\;\; A_{\chi}=\{f\in A: \sigma(f) = \chi(t)f\},
$$
where the direct sum is indexed by the characters $\chi\in \widehat{T}_r$. This is essentially equivalent to the regularity of $\sigma$.

\smallskip

We are interested in defining the above analogue torus actions on the free associative algebra, rather than the affine space, as a generalization of the above discussions. Note that we are interested only in the case of finitely many free generators.

%We are interested in the free associative algebra generalization of the above picture. %As always, the free associative algebra $F_n(\mathbb{K})\equiv F_n = \mathbb{K}\langle z_1,\ldots, z_n\rangle$ is the free $\mathbb{K}$-module generated by the free monoid $\{z_1,\ldots, z_n\}^*$ of \emph{generators} with multiplication given by $\mathbb{K}$-linearly extended concatenation.

%\begin{Def}\label{deffreealgebra}
%The free associative algebra $F_n(\mathbb{K}) = F_n$ in $n$ generators $z_1,\ldots, z_n$ is given by the $\mathbb{K}$-module
%$$
%F_n = \mathbb{K}\langle z_1,\ldots, z_n\rangle
%$$
%of words $z_{i_1}z_{i_2}\ldots z_{i_s}$ in the alphabet $\{z_1,\ldots, z_n\}$ (the free generators) together with the %associative algebra multiplication given by the $\mathbb{K}$-linearly extended word concatenation.
%\end{Def}

\smallskip

The concept of \emph{regular} torus action admits the following generalization to the case of the free associative algebra.

\begin{Def}[Regular action on free algebra]\label{defactionregularfree}
A regular action of an $r$-torus on the free associative algebra $F_n(\mathbb{K})$ is a $\mathbb{K}$-algebra homomorphism
$$
\sigma: F_n(\mathbb{K})\rightarrow F_n(\mathbb{K})\otimes \mathbb{K}[t_1,\ldots, t_r, t_1^{-1},\ldots, t_r^{-1}]
$$
with the property that every element $\sigma(f)(t)$ is uniquely a sum of the form $\sum \chi(t)f_{\chi}$.
\end{Def}

In the sequel, we employ the following short-hand notation for
monomials in the free algebra. The powers of a generator $z_i=z$  are
defined in an obvious manner. Any monomial $z_{i_1}z_{i_2}\ldots z_{i_k}$ can
then be written in a reduced form where subwords  $zz\ldots z$ are
replaced by powers.

We then write
\begin{equation}\label{eq1}
z^I = z_{j_1}^{i_1}z_{j_2}^{i_2}\ldots z_{j_k}^{i_k}
\end{equation}
where by $I$ we mean an assignment of $i_k$ to the subscript $j_k$
in the word $z^I$. Sometimes we refer to $I$ as a ``multi-index,''
although the terminology is not entirely accurate. If $I$ is such a
multi-index,  $|I|$ is defined as the sum $i_1+\cdots+ i_k$.
\smallskip

In \cite{BB1}, Bia\l{}ynicki-Birula proved the following two
theorems, for $\mathbb{K}$ an algebraically closed field.

\begin{thm}(\cite[Theorem 1]{BB1}) \label{BBthm1}
Any regular action of $\mathbb{T}_n$ on $\mathbb{A}^n$ has a fixed
point.
\end{thm}

\begin{thm}(\cite[Theorem 2]{BB1}) \label{BBthm2}
Any effective and regular action of $\mathbb{T}_n$ on $\mathbb{A}^n$
is a representation in some coordinate system.
\end{thm}

%The term ``regular'' is  understood here in the algebraic-geometric
%context of ``regular function.'' (Bia\l{}ynicki-Birula also
%considered birational actions.)

%\smallskip

%In the following section (dedicated to the proof of the free algebra version of Theorems \ref{BBthm1} and \ref{BBthm2}), the ground field is algebraically closed.

\medskip

As mentioned above, an algebraic group action on $\mathbb{A}^n$ is essentially the same as the corresponding action by automorphisms on the algebra
$\mathbb{K}[x_1,\ldots,x_n]$
of coordinate functions. In other words, it is a group homomorphism
$$
\sigma: \mathbb{T}_n\rightarrow \Aut \mathbb{K}[x_1,\ldots,x_n],
$$
and is effective if and only if $\Ker\sigma = \lbrace 1\rbrace$.

The polynomial algebra is a quotient of the free algebra
$$
F_n = \mathbb{K}\langle z_1,\ldots,z_n\rangle
$$
by the commutator ideal $I$ (the two-sided ideal generated by all
elements of the form $fg-gf$). The definition of torus action on the
free algebra is thus purely algebraic.

In this paper we establish the free algebra version of the
Bia\l{}ynicki-Birula theorem, formulated as follows:
\begin{thm} \label{BBfree}
Any effective action $\sigma$ of the algebraic $n$-torus
$\mathbb{T}_n$ on the free algebra $F_n$  is linearizable.
\end{thm}

\smallskip

The linearity (or linearization) problem, as it has become known
since Kambayashi, asks whether all (effective, regular) actions of a
given type of algebraic groups on the affine space of given
dimension are linearizable. According to Theorem \ref{BBfree}, the
linearization problem extends to the noncommutative setting. Several
 results concerning the (commutative) linearization problem are
summarized below.

\begin{enumerate}
    \item Any effective regular torus action on $\mathbb{A}^2$ is linearizable (Gutwirth \cite{RefB}).
    \item Any effective regular torus action on $\mathbb{A}^n$ has a fixed point (Bia\l{}ynicki-Birula (\cite[Theorem 1]{BB1})).
    \item Any effective regular action of $\mathbb{T}_{n-1}$ on $\mathbb{A}^n$ is linearizable (Bia\l{}ynicki-Birula (\cite[Theorem]{BB2})).
    \item Any (effective, regular) one-dimensional torus action (i.e., action of $\mathbb{K}^{\times})$ on $\mathbb{A}^3$ is linearizable (Koras and Russell \cite{KoRu2}).
   \item $\mathbb C^*$-actions on $\mathbb C^3$ are linearizable (Kaliman, Koras,  Makar-Limanov, and
   Russell \cite[Linearization Theorem]{KMMLR}).
    \item If the ground field is not algebraically closed, then a torus action on $\mathbb{A}^n$ need not be linearizable.
    In \cite{Asanuma}, Asanuma proved that over any field $\mathbb{K}$, if there exists a non-rectifiable closed embedding from $\mathbb{A}^{m}$ into $\mathbb A^{n}$, then there exist non-linearizable effective actions of $(\mathbb{K}^{\times})^r$ on $\mathbb A^{1+n+m}$ for $1\le r\le 1+m$.
    \item When $\mathbb {K}$ is infinite and has positive characteristic, there are examples of non-linearizable torus actions on $\mathbb{A}^{n}$ (Asanuma \cite{Asanuma}).
\end{enumerate}

\begin{remark} \label{nonessential1}
A closed embedding $\iota:\mathbb A^m\to\mathbb A^n$ is said to be rectifiable if it is conjugate to a linear embedding by an automorphism of $\mathbb A^n$.
\end{remark}

As can be inferred from the above review, the context of the
linearization problem is rather broad, even for torus actions. The
regulating parameters are the dimensions of the torus and the affine
space. This situation is due to the fact that the general
linearization conjecture (i.e., the conjecture that states that any
effective regular torus action on any affine space is linearizable)
has a negative answer.

\section{Maximal torus action on the free algebra}

In this section, we provide the proof of the free algebra version
(Theorem \ref{BBfree}) of the Bia\l{}ynicki-Birula theorem
\cite{BB1}. The proof proceeds along the lines of the original proof
of Bia\l{}ynicki-Birula in the commutative case.

If $\sigma$ is the effective action of Theorem \ref{BBfree}, then for each $t\in \mathbb{T}_n$ the automorphism
$$
\sigma(t): F_n\rightarrow F_n
$$
is given by the $n$-tuple of images of the generators $z_1,\ldots,z_n$ of the free algebra:
$$
(f_1(t,z_1,\ldots,z_n),\ldots,f_n(t,z_1,\ldots,z_n)).
$$
%Each of  $f_1,\ldots, f_n$ is a polynomial in the free variables.

\begin{Def}\label{deftideal}
If $A$ is an associative $\mathbb{K}$-algebra, then a two-sided ideal $I$ of $A$ is called a $T$\textbf{-ideal}, if it is stable under
all $\mathbb{K}$-algebra endomorphisms of $A$.
\end{Def}

\begin{lem} \label{fixedorigin0} If $I$ is a $T$-ideal of $F_n$ then any automorphism $\varphi: F_n\rightarrow F_n$ induces an automorphism $\bar \varphi: F_n/I \rightarrow F_n/I$ by $\bar \varphi (f+I)=\bar \varphi (f)+I.$
\end{lem}
\begin{proof}
This is immediate since $\varphi (I) \subseteq I$ by hypothesis.
\end{proof}

\begin{lem} \label{fixedorigin}
There is a translation of the free generators
$$
(z_1,\ldots,z_n)\rightarrow (z_1-c_1,\ldots,z_n-c_n),\;\;(c_i\in\mathbb{K})
$$
such that (for all $t\in\mathbb{T}_n$) the polynomials
$f_i(t,z_1-c_1,\ldots,z_n-c_n)$ have constant term 0.
\end{lem}
\begin{proof}
This is a direct consequence of Theorem \ref{BBthm1}. Indeed, the
commutator ideal $I$ is a $T$-ideal so, by Lemma \ref{fixedorigin0}
any action $\sigma$ on the free algebra induces an action
$\bar{\sigma}$ on the commutative algebra
$\mathbb{K}[x_1,\ldots,x_n]$. If $\sigma$ is regular, then so is
$\bar{\sigma}$. By Theorem \ref{BBthm1}, $\bar{\sigma}$ (or rather,
its geometric counterpart) has a fixed point $c_1, \dots, c_n)$.
Therefore the images $\bar g_i$ of the commutative generators $\bar
x_i-c_i$ under $\bar{\sigma}(t)$ (for every $t$) will be polynomials
with constant term 0. Consequently, the  images of the generators $
x_i-c_i$ under $\sigma$ will in $g_i +I$, which also has constant
term 0.
\end{proof}

We may then suppose, without loss of generality, that the polynomials $f_i$ have the form
$$
f_i(t,z_1,\ldots,z_n)=\sum_{j=1}^{n}a_{ij}(t)z_j +
\sum_{j,l=1}^{n}a_{ijl}(t)z_jz_l + \sum_{k=3}^{N}\sum_{J:
\,|J|=k}a_{i,J}(t)z^J
$$
where by $z^J$ we denote, as in \eqref{eq1}, a specific monomial
$$
z_{i_1}^{k_1}z_{i_2}^{k_2}\ldots
$$
(a word in the alphabet $\lbrace z_1,\ldots, z_n\rbrace$ in the
reduced notation; $J$ is the multi-index in the sense described
above); also, $N$ is the degree of the automorphism (which is
finite) and the $a_{ij}, a_{ijl},a_{i,J}$ are polynomials in
$t_1,\ldots, t_n$.

As $\sigma_t$ is an automorphism, the matrix $[a_{ij}]$ that
determines the linear part is nonsingular. Therefore, without loss
of generality we may assume it to be diagonal (just as in the
commutative case \cite{BB1}) of the form
$$
\diag(t_1^{m_{11}}\ldots t_n^{m_{1n}},\ldots, t_1^{m_{n1}}\ldots t_n^{m_{nn}}).
$$

Now, just as in \cite{BB1}, we have the following
\begin{lem} \label{lem1}
The  matrix $[m_{ij}]$ of exponents is nonsingular.
\end{lem}
\begin{proof}
Consider a linear action $\tau$ defined by
$$
\tau(t):(z_1,\ldots, z_n)\mapsto (t_1^{m_{11}}\ldots t_n^{m_{1n}}z_1,\ldots, t_1^{m_{n1}}\ldots t_n^{m_{nn}}z_n),\;\; (t_1,\ldots,t_n)\in\mathbb{T}_n.
$$
If $\mathbb T_1\subset \mathbb T_n$ is any one-dimensional torus, the restriction of
$\tau$ to $\mathbb{T}_1$ is nontrivial. Indeed, were it to happen
that for some $\mathbb{T}_1$,
$$
\tau(t)z=z,\;\;t\in \mathbb{T}_1,\;\;(z=(z_1,\ldots,z_n)),
$$
then our initial action $\sigma$, whose linear part is represented
by $\tau$, would be the identity modulo terms of degree $>1$:
$$
\sigma(t)(z_i) = z_i + \sum_{j,l}a_{ijl}(t)z_jz_l+\cdots.
$$
Now, the equality $\sigma(t^2)(z)=\sigma(t)(\sigma(t)(z))$ implies
\begin{align*}
\sigma(t)(\sigma(t)(z_i))&=\sigma(t)\left(z_i+\sum_{jl}a_{ijl}(t)z_jz_l+\cdots\right)
\\&= z_i+ \sum_{jl}a_{ijl}(t)z_jz_l+\sum_{jl}a_{ijl}(t)(z_j
z_l+\sum_{km}a_{jkm}(t)z_kz_m+\cdots)\\&\hspace*{0.51cm}(z_i+\sum_{k'm'}a_{lk'm'}(t)z_{k'}z_{m'}+\cdots)+\cdots\\&=
z_i+\sum_{jl}a_{ijl}(t^2)z_jz_l+\cdots
\end{align*}
which means that
$$
2a_{ijl}(t)=a_{ijl}(t^2)
$$
and therefore $a_{ijl}(t)=0$. The coefficients of the higher-degree
terms are treated by induction (on the total degree of the
monomial). Thus
$$
\sigma(t)(z) = z,\;\;t\in\mathbb{T}_1,
$$
which is a contradiction since $\sigma$ is effective. Finally, if $[m_{ij}]$ were singular, then one would easily find
 a one-dimensional torus such that the restriction of $\tau$ would be trivial.
\end{proof}

Consider the action
$$
\varphi(t) = \tau(t^{-1})\circ\sigma(t).
$$
The images under $\varphi(t)$ are
$$
(g_1(z,t),\ldots, g_n(z,t)),\;\;(t = (t_1,\ldots,t_n))
$$
with
$$
g_i(z,t) = \sum g_{i,m_1\ldots m_n}(z)t_1^{m_1}\ldots t_n^{m_n},\;\;m_1,\ldots, m_n\in\mathbb{Z}.
$$
Define $G_i(z) = g_{i,0\ldots 0}(z)$ and consider the map
$\beta:F_n\rightarrow F_n$, given by
$$
\beta:(z_1,\ldots,z_n)\mapsto (G_1(z),\ldots, G_n(z)).
$$
\begin{lem} \label{lem2}
$\beta\in \Aut F_n$ and
$$
\beta = \tau(t^{-1})\circ\beta\circ\sigma(t).
$$
\end{lem}
\begin{proof}
This lemma mirrors the final part in the proof in \cite{BB1}. The conjugation is straightforward, since for every $s,t\in\mathbb{T}_n$ one has
$$
\varphi(st) = \tau(t^{-1}s^{-1})\circ\sigma(st) =
\tau(t^{-1})\circ\tau(s^{-1})\circ\sigma(s)\circ\sigma(t)=\tau(t^{-1})\circ\varphi(s)\circ\sigma(t),
$$
and $\beta $ is the constant part of $\varphi$ in terms of $t$, so comes out the same for $st$ as for $t$.

Denote by $\widehat{F}_n$ the power series completion of the free algebra $F_n$, and let $\hat{\sigma}$, $\hat{\tau}$ and $\hat{\beta}$ denote the endomorphisms of the power series algebra induced by the corresponding morphisms of $F_n$. The endomorphisms
$\hat{\sigma}$, $\hat{\tau}$, $\hat{\beta}$ come from (polynomial) automorphisms and therefore are invertible.

Let
$$
B_i(z) := \hat{\beta}^{-1}(z_i) =  \sum_{J}b_{i,J}z^J
$$
($z^J$ again denoting the monomial with multi-index $J$). Then
$$
\hat{\beta}\circ\hat{\tau}(t)\circ\hat{\beta}^{-1}(z_i) = B_i(t_1^{m_{11}}\ldots t_n^{m_{1n}}G_1(z),\ldots,t_1^{m_{n1}}\ldots t_n^{m_{nn}}G_n(z)).
$$
Now, from the conjugation property we must have
$$
\hat{\beta}=\hat{\sigma}(t^{-1})\circ\hat{\beta}\circ\hat{\tau}(t),
$$
and therefore $\hat{\sigma}(t) =
\hat{\beta}\circ\hat{\tau}(t)\circ\hat{\beta}^{-1}$ and
$$
\hat{\sigma}(t)(z_i) = \sum_{J}b_{i,J}(t_1^{m_{11}}\ldots t_n^{m_{1n}})^{j_1}\ldots (t_1^{m_{n1}}\ldots t_n^{m_{nn}})^{j_n}G(z)^J;
$$
here the notation $G(z)^J$ stands for a word in $G_i(z)$ with multi-index $J$, while the exponents $j_1,\ldots, j_n$ count how many times a given index appears in $J$ (or, equivalently, how many times a given generator $z_i$ appears in the word $z^J$).

Therefore, the coefficient of $\hat{\sigma}(t)(z_i)$ at $z^J$ has the form
$$
b_{i,J}(t_1^{m_{11}}\ldots t_n^{m_{1n}})^{j_1}\ldots (t_1^{m_{n1}}\ldots t_n^{m_{nn}})^{j_n}+S
$$
with $S$ a finite sum of monomials of the form
$$
c_L (t_1^{m_{11}}\ldots t_n^{m_{1n}})^{l_1}\ldots (t_1^{m_{n1}}\ldots t_n^{m_{nn}})^{l_n}
$$
with $(j_1,\ldots,j_n)\neq (l_1,\ldots,l_n)$. Since the  matrix
$[m_{ij}]$  of exponents  is nonsingular, if $b_{i,J}\neq 0$, we can
find a $t\in\mathbb{T}_n$ such that the coefficient is not 0. Since
$\sigma$ is a regular action, the degree
$$
\sup_{t}\Deg (\hat{\sigma})
$$
is a finite integer $N$ (the degree is bounded globally with respect to $t$). Together with the previous statement, this
implies that
$$
b_{i,J} = 0,\;\;\text{whenever}\;\;|J|>N.
$$
Therefore, $B_i(z)$ are polynomials in the free variables. What remains is to notice that
$$
z_i = B_i(G_1(z),\ldots,G_n(z)).
$$
Thus $\beta$ is an automorphism.
\end{proof}
From Lemma \ref{lem2} it follows that
$$
\tau(t) = \beta^{-1}\circ \sigma(t)\circ\beta
$$
which is a linearization of $\sigma$. Theorem \ref{BBfree} is
proved.

%\section{Action of $\mathbb{K}^{\times}$ on $\mathbb{K}\langle z_1,z_2\rangle$}

\section{\texorpdfstring{Action of $\mathbb K^{\times}$ on $\mathbb K\langle z_1,z_2\rangle$}{Action of K* on K<z1,z2>}} \label{S:4.:::}

The proof of linearity in the case of a maximal torus action on the
free algebra is obtained from the original proof of
Bia\l{}ynicki-Birula in a straightforward manner. Having done that,
one could try to prove the free version of the other result of
Bia\l{}ynicki-Birula, \cite{BB2}, on the linearity of the action of
$\mathbb{T}_{n-1}$. However, taking that path, one quickly runs into
trouble. It is not implausible, however, that the free associative version of the main
result of \cite{BB2} holds. We have  the following conjecture.

\begin{conj} \label{bbfreeconj}
Any effective regular action of $\mathbb{T}_{n-1}$ on the free
algebra $F_n(\mathbb{K})$ is linearizable, for $\mathbb{K}$
algebraically closed.
\end{conj}

Despite the lack of a strategy to attack Conjecture \ref{bbfreeconj}
in its entirety, a minor comment can be made right away. Namely, we can
state the result of Conjecture \ref{bbfreeconj} for the exceptional case $n=2$ as
a consequence of the \emph{tameness} of polynomial algebra and free algebra automorphisms in two variables -- a result due to Jung -- Van der Kulk \cite{Jung, VdK} and Makar-Limanov \cite{ML2} (see
also \cite{Czer}). The situation is as follows.

An automorphism $\varphi\in\Aut\mathbb{K}[x_1,\ldots,x_n]$ is referred to as \emph{elementary} if it is of the form
$$
\varphi = (x_1,\ldots,\;x_{k-1},\;ax_k+f(x_1,\ldots,x_{k-1},\;x_{k+1},\;\ldots,\;x_n),\;x_{k+1},\;\ldots,\;x_n)
$$
with $a\in\mathbb{K}^{\times}$. Note that linear invertible changes of variables -- that is, transformations of the form
$$
(x_1,\;\ldots,\;x_n)\mapsto (x_1,\;\ldots,\;x_n)A,\;\;A\in\GL(n,\mathbb{K})
$$
are realized as compositions of elementary automorphisms.

The subgroup $\TAut\mathbb{K}[x_1,\ldots,x_n]$ of $\Aut\mathbb{K}[x_1,\ldots,x_n]$ generated by all elementary automorphisms is the group of so-called \emph{tame automorphisms}. Automorphisms that are not tame are called \emph{wild}. For $n>2$ there exist wild automorphisms -- the most famous example given by the Nagata automorphism ($n=3$, $\Char \mathbb{K}=0$):
$$
(x_1+(x_1^2-x_2x_3)x_1,\;x_2+2(x_1^2-x_2x_3)x_1+(x_1^2-x_2x_3)^2x_3,\;x_3)
$$
as demonstrated by Shestakov and Umirbaev \cite{Shes2}. Furthermore, for the free algebra $F_3$ case ($\Char \mathbb{K}=0$), Umirbaev \cite{Umir} proved that the Anick automorphism
$$
(x_1 + x_3(x_1x_3 - x_3x_2), x_2 + (x_1x_3 - x_3x_2)x_3, x_3)
$$
is wild.

\smallskip

In the planar case $n=2$, on the other hand, all automorphisms of both the commutative polynomial algebra and the free associative algebra (over the ground field of any characteristic) are known to be tame. This statement is essentially an accumulation of results of Jung (\cite{Jung}, polynomial algebra, characteristic zero), Van der Kulk (\cite{VdK}, polynomial algebra, positive characteristic) and Makar-Limanov (free algebra and the general case). In particular, the following theorem is true.

\begin{thm}[Makar-Limanov, \cite{ML2}]\label{MLthm}
Let $\mathbb{K}$ be a field. Then the homomorphism
$$
\Phi: \Aut \mathbb{K}\langle z_1,z_2\rangle \rightarrow \Aut \mathbb{K}[x_1,x_2]
$$
induced by abelianization (i.e. for an automorphism $\varphi\in \Aut
\mathbb{K}\langle z_1,z_2\rangle$, the polynomials
$\Phi(\varphi)(x_i)$ that define the image under $\Phi$ are images
of $\varphi(z_i)$ under the projection map $\mathbb{K}\langle
z_1,z_2\rangle\rightarrow \mathbb{K}\langle z_1,z_2\rangle / I$, for
$I$ the commutator ideal) is an isomorphism.
\end{thm}

\begin{remark}\label{remlift}
The inverse $\Theta$ of the isomorphism $\Phi$ is called the lifting
map.
\end{remark}

As a corollary of Makar-Limanov's theorem, we immediately obtain the following result.

\begin{thm}\label{BBfree2}
Let $\mathbb{K}$ be algebraically closed. Any effective regular action of the one-dimensional torus $\mathbb{G}_m$ on the free algebra $\mathbb{K}\langle z_1,z_2\rangle$ is linearizable.
\end{thm}
\begin{proof}
Straightforward: Makar-Limanov's theorem guarantees that the abelianization homomorphism $\Phi$ is an isomorphism in the case $n=2$, which implies that if $\sigma$ is an effective regular action of $\mathbb{G}_m$ yields an action (given by $\Phi(\sigma)$) that is also effective and regular. The latter therefore is linearizable due to Bia\l{}ynicki-Birula's theorem \cite{BB2}, at which point one obtains the linearizing change of the free generators by lifting the Bia\l{}ynicki-Birula linearization with $\Theta$.
\end{proof}

In fact, the situation described above can be generalized to an $r$-torus of arbitrary dimension, conditioned on the linearity (linearizability) on the image of the action with respect to abelianization.

%provided that the linearization assumption on the result of abelianization is present. This observation motivates the following conjecture.

\begin{conj}\label{conjgentoruslinfree}
An effective regular action of $\mathbb{T}_r \simeq (\mathbb{G}_m)^r$ on $\mathbb{K}\langle z_1,\ldots, z_n\rangle$ is linearizable if and only if the action
$$
\mathbb{T}_r\rightarrow \Aut \mathbb{K}[x_1,\ldots, x_n]
$$
induced by abelianization is linearizable.
\end{conj}

At the moment, we are unaware of any counterexamples to this conjecture.

\section{Positive-root torus actions}

For a given action $\phi$ of $r$-dimensional torus $\mathbb{T}_r$ on the (commutative) polynomial algebra, let $\phi_1$ denotes its linear part, i.e., the mapping constructed from degree one components of the images $\phi(t_1,\ldots, t_r, x_i)$. If the action is regular, then the eigenvalues of $\phi_1$ are
$$
\lambda_i = \prod_j t_j^{k_{ij}},
$$
with $k_{ij}$ integers.

Following \cite{BB3}, we have the definition.

\begin{Def}\label{defposroot}
The action $\phi$ is a positive-root action if all
$k_{ij}$ are non-negative integers and for every $i$ there exists $j$ such that $k_{ij}\neq 0$.
\end{Def}

\smallskip

In this section, we consider positive-root torus actions and prove the linearity property analogous to the Bia\l{}ynicki-Birula
theorem. Positive-root actions were studied in \cite{KR,KMMLR} (under a slightly different name). In particular, the following linearization theorem for the  effective regular positive-root action on the polynomial algebra is given in \cite{KR}.
\begin{thm}[\cite{KR}]\label{thmkr}
Any effective regular positive-root action of $\mathbb{T}_r$ on $\mathbb{K}[x_1,\ldots, x_n]$ is linearizable.
\end{thm}

The assumption on the actions turns out to be strong enough for the linearization to be achieved regardless of the torus's
dimension. We employ the notion of generic matrices and use reduction to generic matrices, which amounts to inducing
(commutative) polynomial mappings and positive-root torus actions on commutative polynomial algebras, where linearization may be used to arrive at linearization of the initial action.

\smallskip

More precisely, we establish the following theorem, which is the free associative analog of the Theorem \ref{thmkr} and is regarded the main result of the present section.
\begin{thm}\label{thmkrfree}
 Any regular torus action $\sigma:\mathbb{T}_r\times F_n\to F_n$  with positive roots  is linearizable.
\end{thm}

As was mentioned in the introduction, direct adaptation of proofs in the commutative category to the free associative case (as well as other associative algebras) has its limitations. Nevertheless, sometimes imposition of some additional assumptions may pave the way for a novel proof!

We discuss two potential approaches on the proof of Theorem \ref{thmkrfree}. In both directions the candidate for the conjugation mapping is given by the ``eigenstate'' endomorphism, but they differ somehow in the way they establish invertibility. The first approach associates to the hypothesized conjugation mapping $\phi$ a submodule of the free bimodule of K{\"a}hler differentials of the free algebra, after which certain tools from general algebra are used to demonstrate
that this submodule coincides with the entire bimodule. Note well that this approach which is showed in Subsection \ref{generic} hinges on the validity of Conjecture \ref{lemgroeb} regarding the existence of a finite Gr\"obner basis.
\smallskip

The techniques involving variants of the associative normal form, Gr{\"o}bner bases and special systems of (bi)module generators are relevant in this context. The general framework of noncommutative Gr{\"o}bner bases and the normal form was developed by Bergman in \cite{Berg}. The underlying theory has since proven its utility and relevance to many curious and delicate phenomena. It is, for instance, known that the problem of determining whether an element belongs to a two-sided ideal of a free algebra is (algorithmically) undecidable -- on the other hand, left and right ideals in a free algebra admit finite Gr{\"o}bner bases (the latter statement follows, in particular, from the results of the paper \cite{KBTrees}). Also, for algebras admitting finite Gr{\"o}bner bases the \emph{equality problem} is decidable, while the nilpotency and the zero divisors problems are not (cf. \cite{IPMal}). On the other hand, still, for \emph{monomial algebras} (i.e. algebras satisfying monomial identities) the zero divisors problem is decidable, see \cite{BBL} and \cite{KBTrees}. Overall, the exploration of Gr{\"o}bner bases of ideals and modules over free algebras represents an interesting ongoing line of inquiry and in part motivates our considerations in this instance.

In the second approach, we success to prove Theorem \ref{thmkrfree} without any additional conjecture by employing
a result of Yagzhev on rational invertibility of free algebra endomorphisms,
 while polynomial invertibility of a rationally invertible map whose liftings to generic matrices are polynomially invertible
 is established along the lines of \cite{Yag1} and \cite{Yag3}.

\smallskip

Recently, substantial progress in certain areas of affine algebraic geometry has been made, enabled by the utilization of techniques from mathematical physics, deformation quantization in particular \cite{BKK1, K-BK2, Kon, Kon2}. In the papers \cite{BKK1, K-BK2, Kon2} a reduction of the Weyl algebra modulo infinite prime was introduced in order to induce a Poisson algebra structure on the center and, consequently, using ultraproduct constructs, define a homomorphism from Weyl algebra endomorphisms to polynomial symplectomorphisms, essentially by restriction after an ultraproduct embedding (see \cite{K-BK2}, \cite{Tsu1, Tsu2} and \cite{K-BE2} for details). The aforementioned homomorphism distinguishes globally invertible maps and thus establishes an equivalence between the Jacobian Conjecture and the Dixmier Conjecture on Weyl algebra endomorphisms.

\smallskip

Oftentimes, passing to the noncommutative setting provides one with a clearer picture compared to the commutative case. Apart from the aforementioned result on the equivalence between the Jacobian and the Dixmier Conjecture, alternative proofs using quantization have been found for several classical results, such as Bergman's centralizer theorem, the problem of construction of an algebraically closed skew field, and others -- see \cite{KBL} and \cite{Z2} for a review of relevant developments. Reduction to generic matrices and deformation quantization were used in \cite{KRZ, KRZ2} to establish an alternative proof to Bergman's centralizer theorem (cf. also \cite{Raz1, Z1}). Furthermore, some of the methods of this section (as inspired by Yagzhev) have been observed to bear some resemblance to certain generalities described in, for instance, \cite{BKK1} and \cite{GW} (cf. also \cite{E1, Z1}).

\smallskip

The Jacobian Conjecture in its algebraic reformulation has been resolved in the affirmative for certain classes of associative algebras, in particular, for free associative algebras by Yagzhev \cite{Yag3} and for free metabelian algebras by Umirbaev \cite{Umir2}. Yagzhev's proof of the Jacobian Conjecture for the free associative algebra is based on the idea of deformation. If an endomorphism of the free algebra $F_n$ does not admit an inverse (polynomial, in the free generators), then it induces a parametric family of algebra pre-images which is incompatible with the invertibility of the Jacobian. In this context, the connection between the Jacobian problem for the free algebra and the linearity problem for regular positive-root torus actions becomes particularly interesting.

\subsection{Reduction to generic matrices}\label{generic}

In order to prove the free associative version of this theorem, we devise a way to reduce the positive-root case to the commutative
case. To that end, we introduce  generic matrices and induce the action on the rings of coefficients. Let $M_n(\mathbb R)$ denote the ring of
$n$ by $n$ matrices with entries in $\mathbb R$.

\begin{Def}\label{defgenmatr}
A \textbf{generic $m\times m$ matrix}  is a matrix

$$ Y_i:=
\begin{bmatrix}
x_{11} & x_{12} & \ldots & x_{1m}\\
\vdots & \vdots &\ddots & \vdots\\
x_{m1} & x_{m2} & \ldots & x_{mm}
\end{bmatrix}
$$
whose entries are commuting indeterminates $x_{ij}$. The
$\mathbb{K}$-\textbf{algebra $\mathbb{K}\langle Y_1, \dots,
Y_n\rangle_m$ of $n$ generic $m\times m$ matrices} is the algebra
generated over $\mathbb{K}$ by $n$ distinct generic $m\times m$
matrices. It is a subalgebra of the algebra
$$
M_m(\mathbb{K}[  x^{(1)}_{ij}, \ldots,  x^{(n)}_{ij}])
$$
of matrices with  polynomial entries in $x^{(k)}_{ij}$, $1\leq
i,j\leq m$, $1\leq k \leq n$.
\end{Def}
\begin{remark}\label{remark54}
The algebra $\mathbb{K}\langle Y_1, \dots, Y_n\rangle_m$ is a basic object in
the study of the polynomial identities and the invariants of the generic $m \times m$ matrices, which has been extensively utilized by Amitsur in \cite[\S~1.8]{BKR}. There is a canonical surjection from the free algebra $F_n=\mathbb{K}\langle z_1,\ldots, z_n\rangle$ to
$\mathbb{K}\langle Y_1,\dots,Y_n\rangle_m$, that sends the generator $z_i$ to $Y_i$. For an infinite field $\mathbb{K}$, the kernel of this map is the $T$-ideal $I_m$ of the polynomial identities satisfied by the matrix algebra $M_m(\mathbb{K})$, as stated in \cite[Remark~1.8.12]{BKR}. These polynomial identities have a constant term  0. Additional related results can be found elsewhere (cf. for instance \cite{ALB, BBRY,Schof}). The connection between this topic and the Jacobian Conjecture was established by Yagzhev
\cite{Yag1, Yag2}. Furthermore, the free associative analog of the Jacobian Conjecture has been studied in  \cite{Di, DiLev}, where a relevant criterion was formulated and proved.
\end{remark}

%As we have have seen already, the T-ideal is in particular a subrepresentation of $F_n $ which can be seen as a direct sum of monomials $\prod_j z_{\pi(j'j)}^{k_{ij}}$ and if we suppose the generic matrices $t_{k_{ij}}$ to be diagonalizable then each monomial will be identified with $\det(t_{k_{ij}}):t_{k_{ij}}\to \mathbb G_m$.
\smallskip

Let $h$ be an endomorphism of the free associative algebra $F_n$ (over $n$ generators). Then since $I_m$ is a $T$-ideal, then we have $h(I_m)\subseteq I_m$ for all $m$. Hence $h$ induces an endomorphism
$h_{I_m}$ of $F_n/ I_m$. If $h$ is invertible, then $h_{I_m}$ is invertible, but the converse need not be true.

\smallskip
%!! The rest of the section needs to be re-worked.

Now, take $h$ to be an element of the group of transformations
$\Delta_n \subset \End_{\mathbb K}(F_n)$ satisfying  the requirement that its elements have invertible Jacobian and $h(z_i)=\alpha_i z_i+\varphi_i(z_1,\cdots,z_{i-1})$ for
$\alpha_i\in \mathbb{K}^{\times}$ with each $\varphi_i$  a polynomial of order at
least $2$.

In \cite{Di, DiLev, Schof}, it has been shown that the invertibility of the Jacobian
implies invertibility of the mapping $h$. See also \cite{BBRY} for an exposition.

\smallskip

Our argument is based on the techniques referenced in Remark \ref{remark54}  and proceeds as follows. We need to demonstrate that the endomorphism of the free
algebra sending the generators to the eigenvectors of the positive-root action is an automorphism (and thus the desired
 coordinate change for linearizing). To that end, expanding upon the ideas of Yagzhev \cite{Yag1, Yag2} (see also the exposition in \cite{BBRY}), we induce the endomorphism of the algebra of generic matrices $\mathbb{K}\langle Y_1, \dots, Y_n\rangle_m$ (for arbitrary $m$). This in turn induces an endomorphism on the ring of coefficients, which is the commutative polynomial algebra over $nm^2$ variables. The induced mapping corresponds to a positive-root torus action on the commutative algebra, and by Theorem \ref{thmkr} is linearizable via the eigenvector map. Therefore, the induced
mappings themselves are automorphisms. This implies (a nontrivial fact that utilizes techniques from universal algebra) that the
Jacobian of the initial endomorphism is invertible, which together with the free associative Jacobian Conjecture, proved in \cite{Di,DiLev,Schof}, shows invertibility of the initial endomorphism.

\smallskip

\smallskip
To demonstrate the assertion, consider the eigenstate mapping
\begin{equation}\label{eq77}
z_i\mapsto v_i(z_1,\ldots, z_n)
\end{equation}
into the character vectors $v_i$ corresponding to the eigenvalues $\lambda_i$ of the positive-root torus action (the former are read off from the character-indexed decomposition of $F_n$ in accordance with the situation described after Definition \ref{defcharacter}).
The mapping induces an endomorphism
$$
\phi: z_i\mapsto v_i(z_i)
$$
of the free algebra $F_n$.

Replacing $z_1,\ldots,z_n$ by generic $m\times m$  matrices $Y_1,\ldots, Y_n$ and
passing to the coefficient algebras, we construct a polynomial
mapping
$$\Phi_{(m)}:\mathbb{A}_\mathbb{K}^{nm^2} \rightarrow \mathbb{A}_\mathbb{K}^{nm^2}.$$
Note that the ring of coefficients of $m\times m$ generic matrices
viewed in $R=M_m(\mathbb K[x_{ij}^{k}])$ is
$\mathcal{O}(\mathbb{A}_\mathbb{K}^{nm^2}),$ the polynomial ring in
$nm^2$ variables.

Therefore, the mapping $\phi$ induces, for each $m$, an endomorphism
$$
\phi_m: \mathcal{O}(\mathbb{A}_\mathbb{K}^{nm^2})\rightarrow
\mathcal{O}(\mathbb{A}_\mathbb{K}^{nm^2})
$$
of the polynomial algebra.

\smallskip
This mapping corresponds to a positive-root torus action on the
commutative polynomial algebra. By Theorem \ref{thmkr}, it is
invertible and  linearizable.

\smallskip

In order to go back to the invertibility of the initial map, we need to make precise the notion of the Jacobian of an endomorphism of $F_n$. We proceed along the lines of \cite{DiLev}.

Let $F_n$ be the free associative $\mathbb{K}$-algebra. Recall that
the \textbf{module of (K{\"a}hler) differentials}
$\Omega_{F_n;\mathbb{K}}$ of $F_n$ is defined as follows: If
$z_1,\ldots, z_n$ is the set of generators of $F_n$, then
$\Omega_{F_n;\mathbb{K}}$ is   the direct sum
$$
\oplus_{i=1}^{n}F_n \delta(z_i) F_n,
$$
with $\delta(z_i)$ formal elements and
$$
F_n \delta(z_i) F_n\simeq F_n\otimes_{\mathbb{K}}F_n
$$
(the latter bimodule isomorphism is given by $\delta(z_i)\rightarrow
1\otimes 1$). The bimodule $\Omega_{F_n;\mathbb{K}}$ is isomorphic
to the kernel of the map
$$
F_n\otimes F_n\rightarrow F_n,\;\;a\otimes b\mapsto ab,
$$
by means of the identification $\delta(a)\to  a\otimes 1 - 1\otimes
a$ for all $a\in F_n$.

A \textbf{universal $\mathbb{K}$-derivation} is a derivation
$$
\delta: F_n\rightarrow \Omega_{F_n;\mathbb{K}}
$$
which is ``initial'' in the sense that any other derivation
$F_n\rightarrow M$ factors through $\delta$.

The $F_n$-$F_n$-bimodule $F_n\otimes F_n$ can be turned into a
$\mathbb{K}$-algebra by defining  multiplication
$$
(a\otimes b)(c\otimes d) = ac\otimes db
$$
on simple tensors and extending by linearity. The resulting algebra
is denoted by~$F_n^e$ and is called the \textbf{enveloping algebra}
of $F_n$. The free algebra $F_n$ carries the structure of a left
$F_n^e$-module given by
$$
(a\otimes b)c = acb.
$$
The multiplication map $F_n^e\rightarrow F_n$ (given by $a\otimes b
\mapsto ab)$ is an $F_n^e$-module homomorphism, which means that
$\Omega_{F_n;\mathbb{K}}$ is a left ideal of $F_n^e$ with  free left
$F_n^e$-module basis $\lbrace \delta(z_1),\ldots,
\delta(z_n)\rbrace$. The image $\delta(a)$ of every element $a\in
F_n$ can therefore be uniquely written as an $F_n^e$-linear
combination of $\delta(z_i)$. We denote it by
$$
\delta(a) = \sum_{i=1}^n \frac{\partial a}{\partial z_i}\delta(z_i).
$$
The projection maps
$$
\Omega_{F_n;\mathbb{K}}\rightarrow F_n \delta(z_i) F_n\simeq
F_n\otimes F_n
$$
are left $F_n^e$-linear, from which it follows that the maps
$$
\frac{\partial }{\partial z_j }:F_n\rightarrow F_n\otimes F_n
$$
are $\mathbb{K}$-derivations defined by
$$
\frac{\partial z_i}{\partial z_j}=\delta_{ij}\otimes 1.
$$

\smallskip

 For an endomorphism $\phi:F_n\rightarrow F_n$, define its Jacobian matrix as
$$
J(\phi) = \left[\frac{\partial \phi(z_i)}{\partial z_j}\right]\in
M_n(F_n^e).
$$
Every endomorphism $\phi$ gives rise to a submodule of
$\Omega_{F_n;\mathbb{K}}$ spanned by the images $\delta(\phi(z_i))$.
We denote this module by $\mathcal{M}(\phi)$.
%The standard lexicographical order on monomials of $F_n$ induces a natural order on monomials in $\Omega_{F_n;\mathbb{K}}$. One orders $\delta x_i$ in accordance with $x_1>x_2>\ldots > x_n$, then a monomial $a_1\ldots a_k \delta x_i b_1\ldots b_l$ ($a_i, b_j\in \lbrace x_1,\ldots, x_n\rbrace$) is distinguished by the total length $(k+l)$, the number $i$ in $\delta x_i$, the length $k$ of the monomial to the left of $\delta x_i$ and, lastly, lexicographical order of the monomial $a_1\ldots a_k b_1\ldots b_l$.

The  main conjecture can now be stated.
\begin{conj}\label{lemgroeb}
There exists a partial order on words in the module $\mathcal{M}(\phi)$ such that $\mathcal{M}(\phi)$  admits a finite Gr{\"o}bner basis, whose leading terms are of total length $\leq 2\Deg(\phi)$.
\end{conj}

As was mentioned in the introductory remarks to the present Section, it is known that the problem of determining whether an element belongs to a two-sided ideal of a free algebra is undecidable, but on the other hand, left and right ideals in a free algebra admit finite Gr{\"o}bner bases. The proof of the latter statement is contained \cite{KBTrees}). There is an indication that the situation described by Conjecture \ref{lemgroeb} as well as by its generalization, Conjecture \ref{lemgroebgen} below, whose underlying structure is that of modules over a tensor product of free algebras, is largely analogous to the one considered in \cite{KBTrees}.

\smallskip

Let us consider a somewhat more general situation. Let $F_{n,k}$
denote the $k$-th tensor power of $F_n$ and let $\mathcal{M}$ be a finitely generated free $F_{n,k}$--$F_{n,k}$ bimodule.
\begin{conj}\label{lemgroebgen}
If $L$ is a submodule of $\mathcal{M}$, then $L$ admits a finite Gr{\"o}bner basis.
\end{conj}

This statement  would allows one, by passing to the algebra of generic matrices of sufficiently large size $m$, to obtain the
invertibility of the Jacobian matrix of $\phi$ from the invertibility of the induced automorphisms $\phi_m$.

\begin{Cprop}\label{propjacobian}
If the Jacobian matrix of the induced endomorphism at any reduction modulo the T-ideal is invertible, then the Jacobian matrix of the initial endomorphism of the free associative algebra  is invertible.
\end{Cprop}
\begin{proof}
%One needs to prove that the submodule of relations corresponding to the Jacobian has a finite Groebner basis. This result follows from the results of \cite{KBTrees}. Alternatively, one can use the diamond lemma (as in Latyshev \cite{Lat}, see also \cite{Berg}).
%Given an endomorphism $\phi$ of the free algebra $F_n$, the images $D(x_i)$ of the generators under the mapping to the module of differentials span a $F_n$-$F_n$-bimodule. The Jacobian of $\phi$ corresponds to a submodule, which by diamond lemma admits a finite Groebner basis. The maximal length among the monomials that constitute the basis elements is therefore a constant $m$. By passing to generic matrices of size $m>m$ and noting that the submodules corresponding to the Jacobian in the two cases are the same, we conclude that the Jacobian of the endomorphism $\phi$ is invertible.
The assertion follows from the well-known observation of Amitsur and Levitzki   that the minimal degree of an identity (and hence
relation) for the algebra of $m\times m$ matrices is $2m$. (In fact equality holds by the Amitsur-Levitzki theorem \cite{AmLev1}.) Hence the free algebra $F_n=\mathbb{K}\langle z_1,\ldots, z_n\rangle$ is \textbf{locally isomorphic} to the algebra $\mathbb{K}\langle Y_1,\dots, Y_n\rangle_m$ of $m\times m$ generic matrices, in the following sense: A \emph{local isomorphism} between two
$\mathbb{K}$-algebras with polynomial identities is  a $\mathbb{K}$-module isomorphism between the submodules of words of
length $<2m$ ($m$  constant) with the additional requirement that the product of two monomials, each of degree $<m$, is mapped by the $\mathbb{K}$-module isomorphism to the product of their images. In our case, the constant $m$ is given by the size of the generic matrices.

If the endomorphism $\phi$ of $F_n$ is as in the statement, then for every $m$, one can define the module $\mathcal{M}'(\phi_m)$ which is the submodule in the module
$\Omega_{\mathbb{K}\langle
Y_1, \dots, Y_n\rangle_m/\mathbb{K}}$ in the module of differentials of the algebra of generic matrices generated by
$\delta' \phi_m(X_i)$ (where $\delta'$ is the universal derivation for $\mathbb{K}\langle Y_1, \dots, Y_n\rangle_m$). It follows from
Conjecture~\ref{lemgroeb} 
%(or rather from a setup similar to that conjecture) 
that the module $\mathcal{M}'(\phi_m)$ admits a
Gr{\"o}bner basis with leading terms of total length $\leq 2m$. Now, by assumption, the modulo T-ideal reduction $\phi_m$ is an
automorphism, which means that $J(\phi_m)$ is invertible; therefore $\mathcal{M}'(\phi_m)$ must coincide with $\Omega_{\mathbb{K}\langle
Y_1, \dots, Y_n\rangle_m/\mathbb{K}}$. This means that every elementary differential $\delta'(X_i)$ can be reduced to 0 by a
finite number of elements of total length not greater than a constant (independent of $m$). Taking $m$ sufficiently large and using the local isomorphism of the Amitsur - Levitzki theorem, we see that $\delta z_i$ can also be reduced to 0 in $\mathcal{M}(\phi)$, which means that
$$
\delta z_1,\ldots, \delta z_n\in \mathcal{M}(\phi)
$$
and, therefore, $\mathcal{M}(\phi) = \Omega_{F_n;\mathbb{K}}$. It follows that the Jacobian matrix of the endomorphism $\phi$ is
invertible.
%then for large enough $m$ (it suffices to take $m> 2\Deg \phi + 1$ thanks to Lemma \ref{lemgroeb} and the corollary after it), the replacement of $z_i$ by $m\times N$ generic matrices does not impose any identities and therefore produces a Jacobian submodule isomorphic to $\mathcal{M}(\phi)$. The invertibility of the Jacobian matrix of $\phi$ follows then from the assumption on the induced morphisms.
\end{proof}

As a consequence of Conjectural Proposition \ref{propjacobian} and the free associative Jacobian Conjecture \cite{Di, DiLev, Schof}, the mapping $\phi$ is an automorphism. This yields the desired linearization of the positive-root torus action. 

It is important to clarify that while this approach is incomplete due to the unproven Conjecture \ref{lemgroeb}, and this first approach is a heuristic argument that does not lead to a complete proof of Theorem \ref{thmkrfree}. This heuristic approach is included in the paper for motivational purposes only.

\begin{remark}\label{remnegative}
The negative-root torus action, defined analogously to the
positive-root action, with the requirement $k_{ij}>0$ replaced by
$k_{ij}<0$, is also linearizable, which can be seen by composing the
action with group inversion, thus reducing to the positive-root
case.
\end{remark}

\subsection{Rational invertibility and a theorem of Yagzhev}
We now discuss a more direct approach, independent of the assumption
of Conjecture~\ref{lemgroeb}, which is,  as with the previous route,
based on the ideas of Yagzhev \cite{Yag1, Yag2, Yag3, Yag4}.

\medskip

Let $\mathbb{K}$ be an infinite field and let $\varphi$ be an
endomorphism of the free algebra $F_n = \mathbb{K}\langle
z_1,\ldots, z_n\rangle$ given by:
$$
\varphi: (z_1,\ldots, z_n) \mapsto (A_1(z_1,\ldots, z_n),\ldots,
A_n(z_1,\ldots, z_n)).
$$
Let $P$ be a finite-dimensional unital associative
$\mathbb{K}$-algebra. For our purposes $P$ is either the algebra
$\mathcal M_m := \Mat_m(\mathbb{K})$ or a subalgebra of that algebra.

The endomorphism $\varphi$ induces a polynomial mapping $ A_P:
P^n\rightarrow P^n$ given by $$ \;\; (p_1,\ldots, p_n)\mapsto
(A_1(p_1,\ldots, p_n),\ldots, A_n(p_1,\ldots, p_n)),
$$
where $P^n$ denotes the product of $n$ copies of $P$. If one
fixes a (finite) basis in $P$ and rewrites $A_P$ in these
coordinates, one arrives at a polynomial endomorphism~$\varphi_P$ in
$n\dim_{\mathbb{K}}P$ commuting variables.

\smallskip

Let $\Deg$ be the standard $\mathbb{Z}$-grading on $F_n$ (i.e. $\Deg
z_i = 1$ for all $i$), and let $y_i$ be an extra free variable
corresponding to $z_i$. If $f\in F_n$ is a polynomial in
$z_1,\ldots, z_n$, then there is a unique $\hat{f}\in
\mathbb{K}\langle z_1,\ldots, z_n, y_1,\ldots, y_n\rangle$ such that

\begin{equation}\label{eq11}\begin{aligned} f(z_1+y_1,& \ldots, z_n+y_n) = \\ & f(y_1,\ldots, y_n) +
\hat{f}(z_1,\ldots, z_n,y_1,\ldots, y_n) + R(z_1,\ldots,
z_n,y_1,\ldots, y_n),\end{aligned}
\end{equation}
where each monomial appearing in $\hat{f}$ has degree 1 with respect to $z_1,\ldots, z_n$, and the lowest degree in $z_1,\ldots, z_n$  in $R$ is not less than $2$.

With every endomorphism $\varphi$ of $F_n$ given by
$$
\varphi: (z_1,\ldots, z_n) \mapsto (A_1(z_1,\ldots, z_n),\ldots,
A_n(z_1,\ldots, z_n)).
$$
one can associate an endomorphism $\hat{\varphi}\in\End\mathbb{K}\langle z_1,\ldots, z_n, y_1,\ldots, y_n\rangle$, determined by
\begin{gather*}
\widehat{\varphi}: (z_1,\ldots, z_n)\mapsto (\hat{A}_1(z_1,\ldots, z_n, y_1,\ldots, y_n),\ldots,\hat{A}_n(z_1,\ldots, z_n, y_1,\ldots, y_n)),\\
(y_1,\ldots, y_n)\mapsto (y_1,\ldots, y_n).
\end{gather*}

\begin{Def}\label{defjacobiend}
The endomorphism $\hat{\varphi}$ is called the  \textbf{Jacobi endomorphism} of the endomorphism $\varphi$.
\end{Def}

\smallskip

If $P$ is a subalgebra of $\mathcal M_m$, then $A_{\mathcal M_m}(P^n)\subseteq P^n$; also  obviously that $A_P$ is the restriction of $A_{\mathcal M_m}$ to $P^n$.

\begin{Def}\label{deftestalgebra}
A subalgebra $P$ of  $\mathcal M_m$ is called a \textbf{test algebra} for the endomorphism $\varphi$ if the set $\mathcal M_m^n\backslash P^n$ is not invariant under $A_{\mathcal M_m}$, i.e. if there exists an element $f\in \mathcal M_m^n\backslash P^n$, such that its image under $A_{\mathcal M_m}$ is in $P^n$.
\end{Def}
If an endomorphism $\varphi$ admits a test algebra, then it cannot be invertible, for otherwise $A_{\mathcal M_m}^{-1}(P^n)\subseteq
P^n$.

\begin{Def}\label{defspecialtestalgebra}
A test algebra $P\subseteq \mathcal M_m$ is called a \textbf{special test algebra} if there exist positive integers $r,k$ with
$m=(r+1)k$, such that $P$ consists of all block matrices of the form
$$
\left[\begin{array}{cccc|c}
\Lambda & 0            & \cdots & 0 & 0\\
0            & \Lambda & \cdots & 0 & 0\\
\vdots     & \cdots     & \cdots & 0 & 0\\
0            & \cdots     & \cdots & \Lambda & 0\\
*  & * & * & * & *
\end{array}\right]
$$
with $\Lambda \in M_k$ occurring $r$ times along the main diagonal.
\end{Def}

In his work \cite{Yag1,Yag3}, Yagzhev considers row-finite matrices, i.e. matrices which contain in each row only a finite number of nonzero entries, whose rows and columns are indexed by positive integers.

\begin{remark}\label{sf}
% In order to circumvent this difficulty,
%we  embed  the free algebra into a skew field which possesses nice
%localization properties. The interested reader may consult Paragraph
%0.8 of P. M. Cohn's textbook \cite{Cohn1}.
We need to embed the free algebra $F_n$ into a skew field. While the straightforward approach poses combinatorial challenges, this problem was nonetheless resolved by Amitsur~\cite{Am}. Jategaonkar~\cite{Ja} subsequently demonstrated that the ring of fractions of any left but not right Noetherian domain contains a free algebra. This result inspired Cohn~\cite{Cohn1} to provide five distinct embeddings.

We shall use Yagzhev's embedding, given in \cite{Yag3}, into the skew field T of fractions of an Ore extension $R[y;
\epsilon]$, where $R=\mathbb{K}(t)$ is the field of rational functions in the variables $\lbrace t = t^{(s)}_p: \; s = 1,\ldots,
n,\; p\in \mathbb{Z}\rbrace$, and $\epsilon$ is the automorphism of the extension $R|\mathbb{K}$ defined by
$$
\epsilon: t^{(s)}_p\mapsto t^{(s)}_{p+1}.
$$
Thus $y t^{(s)}_p = t^{(s)}_{p+1} y.$

 The 1:1 map
$$
\pi: F_n\rightarrow T
$$
defined by
$$
z_{i_1}\cdots z_{i_p}\mapsto  t_1^{(i_1)}\cdots t_p^{(i_p)} y^{p}
$$
is a homomorphism, since $$\pi(z_{i_1}\dots z_{i_p}) =
(t_1^{(i_1)}\dots t_q^{(i_q)} y^{q})( t_1^{(i_{q+1})}\dots
t_{p-q}^{(i_{p})} y^{p-q}) = \pi(z_{i_1}\dots
z_{i_q})\pi(z_{i_{q+1}}\dots z_{i_p}).$$   
The free algebra $F_n$ thus is a subalgebra of T.
\end{remark}

\begin{Def}\label{defrationallyinvertiblefree}
An endomorphism $\varphi\in \End F_n$ is \textbf{rationally
invertible} if every $z_i$ ($i=1,\ldots, n$) is a rational function
of $\varphi(z_1),\ldots, \varphi(z_n)$ in some skew field $\text{T}\supset
F_n$.
\end{Def}

\begin{Def}\label{defrationallyinvertible}
An endomorphism $\varphi\in \End \mathbb{K}[x_1,\ldots, x_n]$ is
\textbf{rationally invertible} if every $x_i$ ($i=1,\ldots, n$) is a
rational function of $\varphi(x_1),\ldots, \varphi(x_n)$ in
$\mathbb{K}(x_1,\ldots, x_n)$.
\end{Def}

The following theorem constitutes the main result of \cite{Yag1} (also compare with \cite{Yag3}, Lemma 3).
\begin{thm}[Yagzhev,\cite{Yag1}]\label{yagthm1}
For every endomorphism $\varphi$ of the free algebra $F_n$ at least one of the following is true:

1. $\varphi$ is rationally invertible and for every finite-dimensional unital associative algebra $P$ the induced polynomial endomorphism $\varphi_P$ is also rationally invertible.

2. $\varphi$ admits a special test algebra.
\end{thm}
%
%Let $\Char \mathbb{K} = 0$.
We have the following important corollary to Yagzhev's theorem.

\begin{cor}[Yagzhev \cite{Yag1}]\label{yagcor} For the free algebra over a field $\mathbb{K}$ of characteristic~0, one has the  equivalence
$$
\varphi\in \Aut \mathbb{K}\langle z_1,\ldots, z_n\rangle
\Leftrightarrow \widehat{\varphi}\in\Aut \mathbb{K}\langle Z_1,\ldots,
Z_n,Y_1,\ldots, Y_n\rangle_m,
$$
where $\widehat{\varphi}$ is the Jacobi endomorphism of $\varphi$ (Definition \ref{defjacobiend}).
\end{cor}

We now consider the endomorphism $\phi$ of the free algebra $F_n$ given by the eigenstate mapping corresponding to the positive root torus action of \eqref{eq77}. The endomorphism $\phi$ can be lifted to the algebra of $n$ generic $m\times m$ matrices, with the lifting denoted by $\phi_m$, as has been done in the previous subsection corresponding to the conjugation (i.e. linearization) of the induced positive-root torus action. The existence of such a linearization is guaranteed by Theorem \ref{thmkr}, and this mapping is invertible. Therefore, $\phi$ does not admit test algebras and by Yagzhev's Theorem~\ref{yagthm1} is rationally invertible, meaning that the pre-image generators $z_1,\ldots, z_n$ can be expressed as rational functions of
$\lbrace\phi(z_i)\rbrace$ in the skew field T (of Remark~\ref{sf}).

The final step is given by the proof of polynomial invertibility of
rationally invertible $\phi$.
The proof is a slight
modification of the proof of the main theorem in \cite{Yag1}.

Suppose that $\phi$ is rationally invertible but not polynomially
invertible; suppose also that all liftings $\phi_m$ of $\phi$ to generic
matrices are invertible.

%As the ring $R[x; \epsilon]$ admits left division (cf. \cite{Yag3}), the images $\phi(z_s)$ are of the form $h_sw_s^{-1}$, with $h_s, w_s\in R[x; \epsilon]$ (note that $h_s$ and $w_s$ are relatively prime: there are $g_s, v_s$ such that $g_s h_s+v_s w_s = 1$). According to our assumption, at least one of the denominators $w_s$ is not equal to $1$.

The first step of the proof is given by the following Lemma, formulated and proved originally by Yagzhev in \cite{Yag3, Yag1} (see also Paragraph 0.8 of \cite{Cohn1}).
\begin{lem}\label{lemdivision}
The images $\phi(z_i)$ are of the form $h_iw_i^{-1}$, with $h_i,
w_i\in R[x; \epsilon]$ (with $h_i$ and $w_i$ relatively prime: there
are $g_i, v_i$ such that $g_i h_i+v_i w_i = 1$). The ring~$R$ is as
in Remark~\ref{sf}.
\end{lem}
Essentially this Lemma is the consequence of the fact that $R[x; \epsilon]$ admits left division.

The main point of the proof is that one can take the quotient of $R[x; \epsilon]$ by cycling out the countable collection of variables $t^{(i)}_m$, for a cycle of sufficient length, in such a way that the endomorphism $\phi$ corresponding to rational
 $\phi(z_i)$ will induce an endomorphism in the reduction which will also be rational (and not polynomial) in the skew field of the quotient.

Fix a sufficiently large positive integer $q$. Let $R[x; \epsilon]_q$ be the quotient of $R[x; \epsilon]$ obtained by identifying $t_q^{(s)}$ with $t_0^{(s)}$, and let $\text{T}_q$ be its skew field of fractions.

Denote for brevity $f_i = \phi(z_i)$. Then, by Lemma \ref{lemdivision}, there exist $h_i, w_i\in R[x; \epsilon]$ such
that $f_i = h_iw_i^{-1}$  in T or, equivalently, $f_i w_i = h_i$. Let $\bar{f}_i$ denote the class of $f_i$ in the quotient $R[x;
\epsilon]_q$.  The following statement is elementary.
\begin{lem} \label{lemxdegree}
If $\Deg_x$ is the degree with respect to $x$, then one has
$$
\Deg_x \bar{f}_i \leq \max (\Deg_x \bar{h}_i, \Deg_x \bar{w}_i).
$$
\end{lem}
(In fact, one sees easily  that $\Deg_x\bar{f}_i = \Deg_x \bar{h}_i
- \Deg_x \bar{w}_i$.)
%Now, as was noted previously, the free algebra $F_n = \mathbb{K}\langle z_1,\ldots, z_n\rangle$ is embedded in $T$ by
%$$
%z_s\mapsto xt^{(0)}_s.
%$$
Then one has the following rather straightforward combinatorial
observation.
\begin{prop}\label{propmaxnumber}
The maximal $k$ such that $t^{(i)}_k$ is in $f_i$ is bounded from
above by
$$
\max(\Deg_x h_i, \Deg_x w_i) + \max\lbrace n: t^{(i)}_n\;\text{is in}\;w_i \;\text{or}\;h_i\rbrace.
$$
\end{prop}
It follows that if $\bar{f}_i$ is not polynomial in T (i.e. if
$w_i$ does not equal $1$), then there is a sufficiently large $q$
such that $\bar{f_i}$ is not polynomial in $\text{T}_q$ -- in other words,
$\phi$ is rationally, but not polynomially, invertible with respect
to $\text{T}_q$. However, as $\text{T}_q$ is finite-dimensional, the presence of a
non-unit denominator in $\phi$ implies the existence of an algebra
of generic matrices such that the lifting of $\phi$ to that algebra will
not be polynomially invertible,  contradicting  the assumption on
$\phi$. The invertibility of $\phi$ is thus established, and Theorem
\ref{thmkrfree} is proved.

\begin{remark}
As noted earlier, the proof presented above is based primarily on the work of A.V. Yagzhev \cite{Yag1, Yag2, Yag3,
Yag4}. The ideas developed in the papers referenced hold valuable
insights into the Jacobian Conjecture and could conceivably provide
a basis for its resolution. Some of the points in Yagzhev's work,
however, seem to require a more detailed analysis. Due to the
connection with the Jacobian Conjecture, as well as in light of
Tsuchimoto's pioneering work in noncommutative algebraic geometry
(\cite{Tsu1, Tsu2, Tsu3}), we believe that a more in-depth review of
Yagzhev's work would be beneficial to the mathematical community.
\end{remark}

\section{Non-linearizable torus actions}\label{sec6}

Examples of non-linearizable torus actions, as well as a method of
studying them, were developed by Asanuma \cite{Asanuma}. It is not
difficult to observe that most of Asanuma's technique of proof  can
be carried to the free associative case without loss. As in
Asanuma's case, the existence of non-linearizable torus actions is
tied to the existence of so-called non-rectifiable ideals in the
appropriate algebras. This section establishes that fact and
discusses some consequences. We will be brief, referring the reader
to the original work \cite{Asanuma} for a more thorough exposition.
One rather remarkable feature of Asanuma's technique is the fact
that, modulo minor details and modifications, it may be repeated
almost verbatim in the associative setting -- a situation similar to
the one we have observed in the Bia\l{}ynicki-Birula's theorem on
the action of the maximal torus. This circumstance is curious, as one might not expect the existence of such a linearizable mapping in general, in light of the negative answer to the automorphism lifting problem
provided in \cite{KBYu}.

\medskip

Let $A$, $B$, $C$, $D$ be (associative) algebras over  a field
$\mathbb{K}$. Take ideals $I\triangleleft A$ and $J \triangleleft
B$.

\begin{Def}\label{defequiv}
The ideal $I$ is  \textbf{equivalent} to $J$ if there exists an
algebra isomorphism $\sigma: A\rightarrow B$ such that
$\sigma(I)=J$. The equivalence relation is denoted by $I\sim J$.
\end{Def}
\begin{Def}\label{defrect}
If $I$ is equivalent to the ideal $\langle z_{\ell+1},\ldots,
z_n\rangle$ of the free algebra $\mathbb{K}\langle z_{1},\ldots,
z_n\rangle$ for some $\ell\leq n$ ($\langle z_{\ell+1},\ldots,
z_n\rangle =0$ when $\ell=n$), then the ideal $I$ is said to be
\textbf{rectifiable}.
\end{Def}
\begin{Def}\label{defrecthom}
An algebra homomorphism $\alpha: A\rightarrow C$ is equivalent to a homomorphism $\beta: B\rightarrow D$, if there exist algebra isomorphisms
 $\gamma$ and~$\delta$ that together with $\alpha$ and~$\beta$ form a commutative diagram:
$$
\begin{tikzcd}
A \arrow{r}{\gamma} \arrow{d}{\alpha} & B\arrow{d}{\beta}\\
C\arrow{r}{\delta} & D
\end{tikzcd}
$$
In particular, if $\alpha$ is equivalent to the projection
$$
\beta:\mathbb{K}\langle z_1,\ldots, z_n\rangle\rightarrow
\mathbb{K}\langle z_1,\ldots, z_{\ell}\rangle
$$
defined by $\beta(z_i)=z_i$ for $i=1,\ldots, \ell$ and
$\beta(z_i)=0$ otherwise, then we call $\alpha$
\textbf{rectifiable}.
\end{Def}
\begin{remark}\label{remker}
Note that if $\alpha \sim \beta$, then the ideals $\Ker\alpha$ and
$\Ker\beta$ are equivalent.
\end{remark}

\begin{lem}\label{lemequiv}
Let  $\alpha,\; \gamma: \mathbb{K}\langle
z_1,\ldots,z_n,y_1,\ldots,y_\ell\rangle\rightarrow C$ be two
surjective $\mathbb{K}$-algebra homomorphisms. If $\alpha(y_j)=0$
($j=1,\ldots, \ell$) and $\gamma(z_i)=0$ ($i=1,\ldots, n$), then
there exist polynomials $f_j\in \mathbb{K}\langle
z_1,\ldots,z_n\rangle$ ($j=1,\ldots, \ell$) and $g_i\in
\mathbb{K}\langle y_1,\ldots,y_\ell\rangle$ ($i=1,\ldots, n$) such
that
$$
\alpha(z_i)=\gamma(g_i),\;\;\alpha(f_j)=\gamma(y_j)
$$

for $i=1,\ldots,n$ and $j=1,\ldots, \ell$. If we set
$\tau=\tau_1\circ\tau_2$ ($\tau_i \in\Aut\mathbb{K}\langle
z_1,\ldots,z_n,y_1,\ldots,y_\ell\rangle$) where

$$
\tau_1: (z_1,\ldots,z_n,\, y_1,\ldots,y_\ell)\mapsto
(z_1+g_1,\ldots,z_n+g_n,\, y_1,\ldots,y_\ell)
$$
and
$$
\tau_2: (z_1,\ldots,z_n,\,,
y_1,\ldots,y_\ell)\mapsto(z_1,\ldots,z_n,\,
y_1-f_1,\ldots,y_\ell-f_\ell)
$$
then
$$
\alpha = \gamma\circ\tau.
$$
In particular, $\alpha$ is equivalent to $\gamma$.
\end{lem}
\begin{proof}
Since the images of $\alpha$ and $\gamma$ coincide,  the polynomials
$g_i$ and $f_j$ exist. As
$$
\gamma\circ\tau_1 : (z_1,\ldots,z_n,y_1,\ldots,y_\ell)\mapsto
(\alpha(z_1),\ldots,\alpha(z_n),\gamma(y_1),\ldots,\gamma(y_\ell)),
$$
it follows that
$$
\gamma\circ\tau_1\circ\tau_2(z_i) =
\gamma\circ\tau_1(z_i)=\alpha(z_i)
$$
and
$$
\gamma\circ\tau_1\circ\tau_2(y_j) = \gamma(y_j)-\alpha(f_j)=0.
$$
Therefore $\gamma\circ\tau=\alpha$.
\end{proof}

\begin{cor}\label{corequivhom}
Let $C$ be a $\mathbb{K}$-algebra generated over $\mathbb{K}$ by
$\ell$ elements and let
$$
\alpha,\;\beta:\mathbb{K}\langle z_1,\ldots, z_n,y_1,\ldots
y_\ell\rangle\rightarrow C
$$
be surjective with
$$
\alpha(y_j)=\beta(y_j)=0
$$
for all $j$. Then $\alpha$ is equivalent to $\beta$. In particular,
if $C = \mathbb{K}\langle y_1,\ldots y_\ell\rangle$, then $\alpha$
and~$\beta$ are both rectifiable.
\end{cor}

\begin{proof}

The algebra $C$ is of the form $\mathbb{K}\langle z_1,\ldots,
z_\ell\rangle$ with $z_j\in\alpha(\mathbb{K}\langle z_1,\ldots,
z_n\rangle)$. Define the homomorphism
$$
\gamma: \mathbb{K}\langle z_1,\ldots, z_n,y_1,\ldots
y_\ell\rangle\rightarrow C
$$
by
$$
\gamma(z_i)=0,\;\;\gamma(y_j)=z_j.
$$
Then $\gamma$ is surjective. By Lemma \ref{lemequiv}, both $\alpha$
and $\beta$ are equivalent to $\gamma$, which implies that $\alpha$
is equivalent to $\beta$.
\end{proof}

\begin{cor}\label{corequivideal}
Let $I$ and $J$ be ideals of the free algebra ${F_n}=\mathbb{K}\langle z_1,\ldots,
z_n\rangle$ such that
$$
{F_n}/I\simeq_{\mathbb{K}} {F_n}/J.
$$
If ${F_n}/I$ is generated over $\mathbb{K}$ by $\ell$ elements as a
$\mathbb{K}$-algebra, then
$$
\langle I, y_1,\ldots, y_\ell\rangle\sim \langle J, y_1,\ldots,
y_\ell\rangle
$$
as ideals of $\mathbb{K}\langle z_1,\ldots, z_n,y_1,\ldots
y_\ell\rangle$. In particular, if ${F_n}/I\simeq \mathbb{K}\langle
z_1,\ldots z_\ell\rangle$, then both $\langle I, y_1,\ldots,
y_\ell\rangle$ and $\langle J, y_1,\ldots, y_\ell\rangle$ are
rectifiable.
\end{cor}
\begin{proof}
 Follows immediately from Corollary \ref{corequivhom}.
\end{proof}

\medskip
\begin{Def}\label{defgraded}
Given an associative $\mathbb{K}$-algebra $A$ and a commutative
monoid $M$, the algebra $A$ is called $M$-graded if it can be
represented as a direct sum
$$
A = \bigoplus_{m\in M}\Gamma_m
$$
of $\mathbb{K}$-modules, such that $\Gamma_{m_1}\Gamma_{m_2}\subseteq \Gamma_{m_1+m_2}$. The map
$$
\Gamma: M\rightarrow \lbrace \Gamma_m\;|\;m\in M\rbrace
$$
is called the $M$-grading of $A$.
\end{Def}
\smallskip

\begin{Def}\label{defrees}
Given an associative $\mathbb{K}$-algebra $A$ and a two-sided ideal
$I$, the (extended) Rees algebra is
$$
\mathcal{R}_A(I) =
A[t,t^{-1};I]=\bigoplus_{n=-\infty}^{+\infty}I^nt^n.
$$ \end{Def}
The Rees algebra is a $\mathbb{Z}$-graded  $\mathbb{K}$-algebra
(according to powers of $t$), as well a $\mathbb{K}[t]$-algebra.

\begin{prop}\label{proprees}
If $\Gamma$ and $\Delta$ denote the Rees algebras 
$\mathcal{R}_A(I)$ and $\mathcal{R}_B(J)$, then the following are
equivalent:

1. $\Gamma\simeq_{\mathbb{K}}\Delta$.

2. $\Gamma\simeq_{\mathbb{K}[t]}\Delta$.

3. $I\sim J$.

(The isomorphisms are graded.)
\end{prop}
\begin{proof}
$(1)\Rightarrow (3)$: Suppose there is a graded $\mathbb{K}$-isomorphism
$$
\sigma: \mathcal{R}_A(I)\rightarrow\mathcal{R}_B(J).
$$
Then, in particular,
$$
\sigma(t^nA) = \sigma(\Gamma_n)=\Delta_n=t^nB
$$
when $n=0$ or $n=1$, and
$$
\sigma(t^{-1}I) = \sigma(\Gamma_{-1}) = \Delta_{-1} = t^{-1}J.
$$
Therefore
$$
\sigma_{|A}:A\rightarrow B
$$
is a $\mathbb{K}$-isomorphism, such that
$$
\sigma_{|A}(I) = \sigma(tAt^{-1}I) = tBt^{-1}J = J,
$$
which realizes the equivalence.

\smallskip

$(3)\Rightarrow (2)$: if $I\sim J$, then there is a $\mathbb{K}$-isomorphism
$$
\theta: A\rightarrow B
$$
such that $\theta(I) = J$. The map $\theta$ extends uniquely to
$$
\theta':A[t,t^{-1}]\rightarrow B[t,t^{-1}],
$$
whose restriction to the Rees algebra (which is a subalgebra of the algebra of Laurent polynomials) furnishes the required $\mathbb{Z}$-graded $\mathbb{K}$-isomorphism.

\smallskip

$(2)\Rightarrow (1)$ is immediate.
\end{proof}

\begin{prop}\label{propreesequiv}
Let $F_n=\mathbb{K}\langle z_1,\ldots, z_n\rangle$ be the free
algebra, and let $I$ and $J$ be two (two-sided) ideals of $F_n$, such
that
$$
F_n/I\simeq F_n/J.
$$
If $F_n/I$ is generated by $\ell$ elements over $\mathbb{K}$, then
$$
\mathcal{R}_{F_n}(I)\langle y_1,\ldots, y_\ell\rangle
\simeq_{\mathbb{K}[t]} \mathcal{R}_{F_n}(J)\langle y_1,\ldots,
y_\ell\rangle.
$$
In particular, if $F_n/I\simeq \mathbb{K}\langle y_1,\ldots,
y_\ell\rangle$, then
$$
\mathcal{R}_{F_n}(I)\langle y_1,\ldots,
y_\ell\rangle\simeq_{\mathbb{K}[t]} \mathbb{K}[t]\langle z_1,
\ldots, z_n, y_1,\ldots, y_\ell\rangle.
$$

\end{prop}

\begin{proof}

Let $B = F_n\langle y_1,\ldots, y_\ell\rangle$, $I' = \langle I,
y_1,\ldots, y_\ell\rangle$ and $J' = \langle J, y_1,\ldots,
y_\ell\rangle$. Then, by Corollary \ref{corequivideal}, $I'\sim J'$,
and by Proposition \ref{proprees}
$$
\mathcal{R}_{B}(I')\simeq_{\mathbb{K}[t]}\mathcal{R}_{B}(J').
$$
But
$$
\mathcal{R}_{B}(I')=\mathcal{R}_{F_n}(I)\langle t^{-1}y_1,\ldots,
t^{-1}y_\ell\rangle
$$
and
$$
\mathcal{R}_{B}(J')=\mathcal{R}_{F_n}(J)\langle t^{-1}y_1,\ldots,
t^{-1}y_\ell\rangle.
$$
As $t^{-1}y_1,\ldots,t^{-1}y_\ell$ are free variables, we must have
$$
\mathcal{R}_{F_n}(I)\langle y_1,\ldots, y_\ell\rangle
\simeq_{\mathbb{K}[t]} \mathcal{R}_{F_n}(J)\langle y_1,\ldots,
y_\ell\rangle.
$$
In particular, when $F_n/I\simeq \mathbb{K}\langle y_1,\ldots,
y_\ell\rangle$, we may take $J = \langle z_{\ell+1},\ldots,
z_n\rangle$, which, together with
$$
\mathcal{R}_{F_n}(J)\langle y_1,\ldots,
y_\ell\rangle\simeq_{\mathbb{K}[t]} \mathbb{K}[t]\langle z_1,
\ldots, z_n, y_1,\ldots, y_\ell\rangle
$$
yields
$$
\mathcal{R}_{F_n}(I)\langle y_1,\ldots,
y_\ell\rangle\simeq_{\mathbb{K}[t]} \mathbb{K}[t]\langle z_1,
\ldots, z_n, y_1,\ldots, y_\ell\rangle.
$$

\end{proof}

\medskip

Any regular action of the $r$-torus $\mathbb{T}_r$ on an associative $\mathbb{K}$-algebra $A$ is equivalent to a homomorphism
$$
\phi: A\rightarrow
A\otimes_{\mathbb{K}}\mathbb{K}[t_1,\ldots,t_r,t_1^{-1},\ldots,t_r^{-1}]=A[t_1,\ldots,t_r,t_1^{-1},\ldots,t_r^{-1}].
$$

Given an element $f\in A$, its image under $\phi$ can be written as
$$
\phi(f) = \sum_{\mathbf m}f_m t_1^{m_1}\ldots t_r^{m_r}
$$
with $\mathbf m=(m_1,\ldots, m_r)\in\mathbb{Z}^r$ corresponding to a character $\mathbf{m}\in\widehat{T}_r$ and $f_{\mathbf m}\in A$. The map $f\mapsto f_{\mathbf m}$ induces a $\mathbb{Z}^r$-grading
$$
\Gamma: \mathbf m\mapsto \Gamma_{\mathbf m} = \lbrace f_{\mathbf m}, \;|\;
f\in A\rbrace
$$
of the algebra $A$.

\begin{Def}\label{defequivactions}
Two (regular) $\mathbb{T}_r$-actions $\phi$ and $\psi$, respectively, on $A$ and $B$ are equivalent, if there exists a $\mathbb{K}$-algebra isomorphism $\sigma: A\rightarrow B$, such that the diagram
$$
\begin{tikzcd}
A[t_1,\ldots,t_r,t_1^{-1},\ldots,t_r^{-1}] \arrow{r}{\sigma\otimes\Id} & B[t_1,\ldots,t_r,t_1^{-1},\ldots,t_r^{-1}]\\
A\arrow{r}{\sigma}\arrow{u}{\phi} & B\arrow{u}{\psi}
\end{tikzcd}
$$
commutes.
\end{Def}

\begin{prop}\label{propequivdeflin}
A regular $\mathbb{T}_r$-action $\phi$ on the free algebra $F_n$ is linearizable
 (in the sense of the previous sections), if and only if it is equivalent, in the sense of Definition \ref{defequivactions}, to a linear action on $F_n$.
 (An action $\psi:F_n\rightarrow F_n[t_1,\ldots,t_r,t_1^{-1},\ldots,t_r^{-1}]$ is  \textbf{linear}
 if the images $\psi(z_i)$ of the generators of $F_n$ are linear in $z_i$).
\end{prop}

\begin{proof}
Straightforward.
\end{proof}

\smallskip

\begin{Def}\label{deffixed}
A two-sided ideal $I$  is   \textbf{fixed} by  an action $\phi$  on
$A$,
 if the image $\phi(I)$ is contained in the ideal
$I[t_1,\ldots,t_r,t_1^{-1},\ldots,t_r^{-1}]$.
\end{Def}
If $I$ is fixed by $\phi$, then $\phi$ induces a canonical $\mathbb{K}$-homomorphism
$$
\overline{\phi}:\gr_I(A)\rightarrow
\gr_I(A)[t_1,\ldots,t_r,t_1^{-1},\ldots,t_r^{-1}]
$$
on the associated graded ring, which therefore defines a torus
action. Then $A/I$ is a subring of $\gr_I(A)$, and we have
$$
\overline{\phi}(A/I)\subset A/I\,
[t_1,\ldots,t_r,t_1^{-1},\ldots,t_r^{-1}].
$$
If $I$ is maximal, then $A/I$ is a simple ring.

The following fact is used by Asanuma.
%By  ``degree''   we mean with
%respect to the induced $\mathbb{Z}^r$-grading, and by ``constant''
%we mean ``contained in the degree zero component''.

\begin{remark} \label{fixed0} It is known, and easy to demonstrate, that if $\mathbb{K}$ is a field then all
invertible elements of $\mathbb{K}[t_1,\ldots,t_r,t_1^{-1},\ldots,t_r^{-1}]$ (where $t_1,\ldots, t_r$ are transcendental variables)
are of the form $at_1^{i_1}\ldots t_r^{i_r}$ with $a\in \mathbb{K}^{\times}$ and $i_1,\ldots, i_r$ integers (the result is true for $R[t,t^{-1}]$ with $R$ a reduced connected commutative ring; broad generalizations to the case $R[G]$ of group algebras for certain types of groups have been obtained in \cite{Neher}). From the characterization of the set of units in $\mathbb{K}[t_1,\ldots,t_r,t_1^{-1},\ldots,t_r^{-1}]$ we obtain the following proposition, which we make use of in the sequel:

-- every subfield of
$\mathbb{K}[t_1,\ldots,t_r,t_1^{-1},\ldots,t_r^{-1}]$ is contained in the degree zero component.
\end{remark}

\begin{prop}\label{factlinassoc1}
If  $A/I$ satisfies a polynomial identity (PI), then $\overline{\phi}(A/I)\subseteq A/I. $
\end{prop}
\begin{proof}
If $A/I$ is a field, then the statement is true by Remark~\ref{fixed0}. In
general, $A/I$ is finite over its center, which must be constant,
but then any algebraic element is constant, implying  $A/I$ is
constant.
\end{proof}

Therefore, for $A/I$ PI, any homogeneous element of $A$ of nonzero
degree (with respect to the $\mathbb{Z}^r$-grading induced by the
torus action $\phi$) is contained in the fixed maximal ideal $I$,
which means in this case that any maximal two-sided ideal $I$ fixed
by $\phi$ is of the form
$$
(I\cap \Gamma_0)\oplus\left(\bigoplus_{m\neq 0}\Gamma_m\right).
$$

\begin{example} 
This argument fails for $A/I$ the Weyl algebra; here $I$ is generated by $z_1 z_2 -1$, and one can define $z_1 \mapsto  z_1t_1$ and $z_2 \mapsto  z_2 t_1^{-1}.$ Then $\phi(A/I)$ is a non-constant image of the Weyl algebra.
\end{example}

 \begin{prop}\label{factlinassoc}
If $\phi$ is a linearizable torus action on the free algebra ${F_n}=\mathbb{K}\langle z_1,\ldots, z_n\rangle$, then there exist a
maximal two-sided ideal $M$ fixed by $\phi$, such that the induced torus action $\overline{\phi}$ on $\gr_M({F_n})$ is equivalent to $\phi$.
\end{prop}
\begin{proof}
Take $M=\langle y_1,\ldots,y_n\rangle_m$, where $y_1, \ldots, y_n$ is the set of linearized generators. Mapping each $y_i \mapsto 0$ sends each $z_i$ to a constant, so the image ${F_n}/M$ is the field $\mathbb{K},$ implying $M$ is a maximal ideal. But $M$ is obviously fixed by $\phi$, and $\phi$ is equivalent to the induced action.
\end{proof}

\smallskip

\begin{prop}\label{propsecmain}
Let $A$ be an associative $\mathbb{K}$-algebra and let $\phi$ be an action of $\mathbb{T}_r$ on $\mathcal{R}_A(I)\langle y_1, \ldots,
y_m\rangle$ defined by the requirement that  $\phi(f)$ is constant for each $f\in A\langle y_1, \ldots, y_m\rangle$, while $\phi(t)$
has some nonzero degree $\mathbf d$ and, given $s$ and all $i=s+1,\, \ldots,\, m$, the degrees of $y_i$ are not in  $\mathbb{Z}^r \mathbf d$. The following assertions are true:
\begin{enumerate}[label=(\arabic*)]
    \item \label{propsecmain:1}Suppose $B$ is another algebra and $\psi$ is a $\mathbb{T}_r$-action on
$\mathcal{R}_B(J)\langle y_1, \ldots, y_m\rangle$, also satisfying the above hypotheses. Then $\phi$ is equivalent to $\psi$ if and only if the ideals $I\langle y_1,\ldots, y_s\rangle$ and $J\langle
y_1,\ldots, y_s\rangle$ are equivalent.
\item \label{propsecmain:2} Let $A = {F_n} = \mathbb{K}\langle z_1,\ldots, z_n\rangle$ be the free
algebra. Then $\phi$ is linearizable if and only if $I\langle y_1,\ldots, y_s\rangle$ is rectifiable.
\end{enumerate}

\end{prop}

\begin{proof}
From the outset we may suppose $s=1$, since
$$
\mathcal{R}_A(I)\langle y_1, \ldots, y_m\rangle = \mathcal{R}_C(IC)\langle y_{s+1}, \ldots, y_m\rangle
$$
with $C = A\langle y_1,\ldots, y_s\rangle$, and therefore we can descend to $C=A$.
\smallskip

Let $R$ and $S$ denote the algebras $\mathcal{R}_A(I)\langle y_1, \ldots, y_m\rangle$ and $\mathcal{R}_B(J)\langle y_1, \ldots, y_m\rangle$, respectively,
and let $\Gamma$ and $\Delta$ be the gradings induced on $R$ and $S$ by the corresponding torus actions.

To prove \ref{propsecmain:1}, it is enough to demonstrate the direction $(\Rightarrow)$, since $(\Leftarrow)$ is obvious. The equivalence of actions means the existence of a graded $\mathbb{K}$-isomorphism
$\bar{\sigma}$. Denote by $P$ and $Q$ the ideals in $R$ and $S$ generated by $y_1, \ldots, y_m$. The ideals $P$ and $Q$ are fixed,
respectively, by $\phi$ and $\psi$. Then, the $\mathbb{Z}^r$ subgroup assumption on the degrees easily implies that $\bar{\sigma}(P)=Q$. Therefore, $\bar{\sigma}$ descends to $\sigma:\mathcal{R}_A(I)\rightarrow \mathcal{R}_B(J)$, and the diagram
$$
\begin{tikzcd}
R\arrow{r}{\bar{\sigma}}\arrow{d}{\pi_A} & S\arrow{d}{\pi_B}\\
\mathcal{R}_A(I)\arrow{r}{\sigma} & \mathcal{R}_B(J)
\end{tikzcd}
$$
commutes ($\pi_A$ and $\pi_B$ are the natural projections $y_i\mapsto 0$). The map $\sigma$ is a $\mathbb{Z}$-graded (with respect to the Rees grading) $\mathbb{K}$-isomorphism, so \ref{propsecmain:1}
follows from Proposition \ref{proprees}.

\smallskip

To prove  \ref{propsecmain:2}, it  again suffices to demonstrate $(\Rightarrow)$.
Suppose $\phi$ is linearizable. By Proposition \ref{factlinassoc}, there exists a two-sided maximal ideal $M$ of $R$ fixed by $\phi$,
such that $A/M \cong \mathbb{K},$  the induced action $\overline{\phi}$ on the associated graded ring $\gr_M(R)$ is equivalent to $\phi$.
The ideal $M$ is generated by the subset
$$
(A\cap M)\, \cup t^{-1}I \, \cup \, \lbrace t,y_1,\ldots,
y_m\rbrace.
$$
Consider the left module $M/M^2$ over the simple ring $R/M$. It is clear from the proof of Proposition \ref{factlinassoc} that $R/M=\mathbb{K}$. The image of the set $\lbrace t,y_1,\ldots, y_m\rbrace$ under the projection
$$
M\rightarrow M/M^2
$$
is thus a linearly independent system over $\mathbb{K}$. Therefore, there exist $f_i\in A\cap M$ ($i=1,\ldots, u$) and $f_{u+j}\in I$ ($j=1,\ldots, n-u$) such that the set of images under the quotient map,
$$
\lbrace \bar{t},\bar{y}_1,\ldots, \bar{y}_m, \bar{f}_1,\ldots, \bar{f}_u,\overline{t^{-1}f_{u+1}}\ldots,\overline{t^{-1}f_{n}}\rbrace
$$
is a basis of the $\mathbb{K}$-vector space $M/M^2$.

Let ${F_n}=\mathbb{K}\langle z_1,\ldots, z_n\rangle$ be the free algebra and let $J = \langle z_{u+1},\ldots, z_n\rangle$. The
isomorphism
$$
\theta: \gr_R(M)\rightarrow S
$$
by
$$
\theta(\bar{t}) = t,\;\theta(\bar{f}_i) =
z_i,\;\theta(\overline{t^{-1}f_{j}}) = t^{-1}z_j,
\;\theta(\bar{y}_k) = y_k
$$
(where $i=1,\ldots, u$, $j=u+1,\ldots,n$, $k=1,\ldots, m$) induces the action $\psi$ on $S$ which fulfills the conditions of \ref{propsecmain:1}.
Therefore, $I$ is equivalent to $J$ and $I\langle y_1,\ldots, y_s\rangle$ is rectifiable.
\end{proof}

This Proposition \ref{propsecmain}, together with Proposition \ref{propreesequiv},
provides the foundation for Asanuma's counter-examples, in our
analogy, this translates to the following statement.

If there exists a non-rectifiable ideal $I$ of the free algebra
$F_n=\mathbb{K}\langle z_1,\ldots, z_n\rangle$ such that
$F_n/I\simeq_{\mathbb{K}} F_m$, then there are examples of
non-linearizable $\mathbb{T}_r$-actions on free associative
algebras.

\smallskip

If $A = \mathcal{R}_{F_n}(I)\langle y_1,\ldots, y_m\rangle$, then by Proposition \ref{propreesequiv}
 there exists an isomorphism of $\mathbb{K}$-algebras
$$
\beta: A\rightarrow F_n[t]\langle y_1,\ldots, y_m\rangle.
$$
If $\phi$ is a (regular) torus action, define $\delta$ via the
commutative diagram
$$
\begin{tikzcd}
A[t_1,\ldots, t_r,t_1^{-1},\ldots,t_r^{-1}]\arrow{r}{\beta\otimes\Id} & F_n[t][t_1,\ldots, t_r,t_1^{-1},\ldots,t_r^{-1}]\langle y_1,\ldots, y_m\rangle\\
A\arrow{r}{\beta}\arrow{u}{\phi} & F_n[t]\langle y_1,\ldots, y_m\rangle\arrow{u}{\delta}.
\end{tikzcd}
$$
Then $\delta$ is equivalent to $\phi$ and is linearizable if and
only if $I\langle  y_1,\ldots, y_s\rangle$ ($s$ is as in Proposition
\ref{propsecmain}) is rectifiable.

 This can be furthered by
obtaining a class of actions which are linearizable if and only if
$I$ itself is rectifiable, at which point the construction of $I$
reduces to lifting the ideals obtained by Asanuma to the free
algebra.

One can replace $F_n[t]\langle
y_1,\ldots, y_m\rangle $ by $\mathbb{K}\langle u, y_1,\ldots, y_m, z_1, \ldots, z_n\rangle$. Summarizing, we obtain

\begin{thm}

  Let $\mathbb{K}$ be an arbitrary field. If for $n>m$ there exists a non-rectifiable epimorphism $\mathbb{K}\langle z_1,\dots,z_n\rangle\to \mathbb{K}\langle z_1,\dots,z_m\rangle $
  then for $1\leq r\leq m+1$ there is a non-linearizable effective regular $\mathbb{T}_r$-action  
  $$\phi: \mathbb{K}\langle u, y_1,\ldots, y_m, z_1, \ldots, z_n\rangle\to \mathbb{K}[t_1,\ldots, t_r, t_1^{-1},\ldots, t_r^{-1}]\langle u, y_1,\ldots, y_m, z_1, \ldots, z_n\rangle.$$

  %Let $\mathbb{K}$ be an arbitrary field. If there exist a non-rectifiable epimorphism $\mathbb{K}\langle z_1,\dots,z_n\rangle\to \mathbb{K}\langle z_1,\dots,z_m\rangle $
  %then there is a non-linearizable effective regular $\mathbb{G}_m^r$-action  $\mathbb{K}\langle t,z_1,\dots,z_n\rangle\to \mathbb{K}\langle z_1,\dots,z_m\rangle
 % $.

\end{thm}

Moreover, it is clear from our generalization of Asanuma's work that
some of his results on the cancelation problem (Section 8 of
\cite{Asanuma}) also have free associative analogs. These are
formulated as follows.

\begin{Def}\label{pRees}
    Given an associative $\mathbb{K}$-algebra $A$ and a (left) ideal $I$, the (left) \textbf{pseudo Rees algebra} $A\langle t,t^{-1}I\rangle$ is the subalgebra in $A\langle t,t^{-1}\rangle$ generated by $A, t, t^{-1}I$.
\end{Def}

The right Rees pseudoalgebra $A\langle t,It^{-1}\rangle$ is defined similarly, but with a right ideal $I$ and is generated by $A$, $t$, and $It^{-1}$.

\begin{lem} \label{lem:leftreesequiv} 
    If (left) ideals $I$ and $J$ in $A$ are equivalent, then $A\langle t,t^{-1}I\rangle\cong_{\mathbb{K}[t]} A\langle t,t^{-1}J\rangle$.
\end{lem}
    The proof is analogue to the one of Proposition \ref{proprees}.

\begin{lem}\label{lem:leftrees} 
Let $I$ be a (left) ideal of $A$, and consider the coproduct $A*k\langle x_1,\dots,x_n\rangle$. Let $I'=(I,x_1,\dots,x_n)$ be the ideal in $A*k\langle x_1,\dots,x_n\rangle$, then
    $$A\langle t,t^{-1}I\rangle*k\langle x_1,\dots,x_n\rangle\cong_{\mathbb{K}[t]}A\langle t,t^{-1}I'\rangle.$$
\end{lem}

\begin{proof}
    The isomorphism can be provided directly by the substitution $x_i\to t^{-1}x_i$.
\end{proof}

\begin{remark}
    Analogous lemmas to Lemmas \ref{lem:leftreesequiv} and \ref{lem:leftrees} can be proved for the right pseudo Rees algebra.
\end{remark}

\begin{Def}\label{defcinv}
Let $D$ be an integral domain. An (associative) $D$-algebra $A$ is
$D$-invariant, if for any $D$-algebra $B$ such that for some $m$ the
free products
$$
A* \mathbb{K}\langle y_1\rangle * \ldots * \mathbb{K}\langle y_m\rangle\simeq_{D} B* \mathbb{K}\langle y_1\rangle * \ldots * \mathbb{K}\langle y_m\rangle
$$
are isomorphic as $D$-algebras, then $A\simeq_{D} B$.
\end{Def}
The main problem of interest is the free associative analog of the
so-called Cancelation Conjecture, as formulated by Drensky and Yu
\cite{DrYu}:

\begin{conj}[Associative Cancellation Conjecture, \cite{DrYu}]\label{conjdryu}
Let $R$ be a $\mathbb{K}$-algebra. If
$$
R*\mathbb{K}\langle y\rangle\simeq_{\mathbb{K}} \mathbb{K}\langle
z_1,\ldots,z_{n-1}, y\rangle
$$
then
$$
R\simeq_{\mathbb{K}}\mathbb{K}\langle z_1,\ldots, z_{n-1}\rangle.
$$
\end{conj}

Asanuma's results on Rees algebras (and their associative
analogs given in Proposition \ref{propreesequiv}) allow us to
establish a version of the Cancelation Conjecture for coproducts
over a (commutative) $\mathbb{K}$-algebra $D$.

The following statement holds.

\begin{thm}\label{thmcancellation}
%Let $D$ be an integral domain which is a $\mathbb{K}$-algebra, and let $x$ be an indeterminate over $D$. 
Given a nonzero element $t\in \mathbb{K}[x]$ and polynomials $f(x)$ and $g(x)$ in $\mathbb{K}[x]$. Let $A=\mathbb{K}[x]\langle t,t^{-1}f(x)\rangle$ and $B=\mathbb{K}[x]\langle t,t^{-1}g(x)\rangle$ be the left pseudo Rees algebras of $(f(x))$ and $(g(x))$ respectively. 

\begin{enumerate}[label=\arabic*)]
    \item \label{thmcancellation:1}If $\mathbb{K}[x]/(f(x))\simeq_{\mathbb{K}} \mathbb{K}[x]/(g(x))$,
    then
    $$
    A * \mathbb{K}\langle y\rangle\simeq B * \mathbb{K}\langle y\rangle.
    $$
   \item \label{thmcancellation:2} If $f(x)\neq g(ax+b)$ for all $a,b\in \mathbb{K}$, then $A\not\cong B$.
\end{enumerate}
\end{thm}

\begin{proof}
\begin{enumerate}[label=\arabic*)]
    \item Let
$$
\mathbb{K}[x]/(f(x))\simeq_{\mathbb{K}} \mathbb{K}[x]/(g(x)).
$$
The element $t$ is transcendental over $\mathbb{K}$, and therefore
by Lemmas \ref{lem:leftreesequiv} and \ref{lem:leftrees} we have\footnote{Or rather, as a
consequence of Proposition \ref{propreesequiv} when commutation relations are imposed on both sides.}
$$
A*\mathbb{K}\langle y\rangle\simeq B*\mathbb{K}\langle y\rangle.
$$
\item We just need to prove that the abelianization of the pseudo Rees algebras $A$ and $B$ are not isomorphic. In fact, the non-isomorphism of two algebras follows from the fact that their commutator factors are not isomorphic, according to Azanuma \cite{Asanuma}.
\end{enumerate}
\end{proof}

Note that, the left pseudo Rees algebras of $(f(x))$ and $(g(x))$ coproduct with a polynomial ring are isomorphic due to previous Theorem \ref{thmcancellation}, therefore, we have the following result.

\begin{thm}
    There exists two non isomorphic $\mathbb{K}$-algebras $A$ and $B$,
    such that  
    $$
    A * \mathbb{K}\langle x\rangle \cong B *\mathbb{K}\langle x\rangle.
    $$
\end{thm}

\section{Discussion}
The noncommutative toric action linearity property has several useful applications. In \cite{KBYu}, it is used to investigate the
properties of the group $\Aut F_n$ of automorphisms of the free algebra. As a corollary of Theorem \ref{BBfree}, one gets

\begin{cor} \label{cor1}
Let $\theta$ denote the standard action of $\mathbb{T}_n$ on $K[x_1,\ldots,x_n]$, i.e., the action
$$
\theta_t: (x_1,\ldots,x_n)\mapsto (t_1x_1,\ldots,t_nx_n).
$$
The lifting  of $\theta$ to an action on the free associative algebra $F_n$  is also given by the standard torus action.
\end{cor}

This assertion plays a part, along with a number of results concerning the induced formal power series topology on $\Aut F_n$,
in the establishment of the free associative analog of Theorem
\ref{commautautthm}.

\smallskip

The proofs in this paper, for the most part, have been extensions of commutative techniques. However, it is a legitimate endeavor to seek proofs of various linearity statements using tools specific to associative algebras, bypassing the known commutative results. As one outstanding example of this problem, we expect the free
associative analog of the second Bia\l{}ynicki-Birula theorem to hold, as formulated in Conjecture~\ref{bbfreeconj}.

Transition to noncommutative geometry presents the investigator with
an even broader context: one now may vary the dimensions as well as
impose restrictions on the action in the form of preservation of the
polynomial identities. Caution is well advised, as some of the results can be generalized in a straightforward manner, as exemplified by the proof in the next section, while others require more subtlety and effort. Of particular note to us, given our ongoing work in deformation quantization (see, for instance, \cite{KGE}), is the following instance of the linearization problem, which we formulate as a conjecture.

\begin{conj} \label{BBsympl}
For $n\geq 1$, let $P_n$ denote the commutative Poisson algebra, i.e. the polynomial algebra
$$
\mathbb{K}[x_1,\ldots,x_{2n}]
$$
equipped with the Poisson bracket defined by
$$
\lbrace x_i, x_j\rbrace = \delta_{i,n+j}-\delta_{i+n,j}.
$$
Then any effective regular action of $\mathbb{T}_n$ by automorphisms of $P_n$ is linearizable.
\end{conj}

A version of Theorem \ref{commautautthm} for the commutative Poisson
algebra is a conjecture of significant interest. It turns out that
the algebra $P_n$ admits a certain augmentation by central variables
which distorts the Poisson structure, such that the automorphism
group of the resulting algebra admits the property of Theorem
\ref{commautautthm}. This case is studied in  \cite{K-BE5}.

\smallskip

Finally, in accordance with the general ideas presented in the present paper, exploring extensions of the generalized Cancellation Conjecture \cite{BMLY, BY} into the associative category becomes an interesting future development. Gupta \cite{Gupta1, Gupta2, Gupta3, Gupta4} has discovered positive-characteristic counterexamples to the Zariski cancellation problem expanding upon Asanuma's result \cite{Asanuma} on the non-linearizable torus actions. In order to investigate the question of existence of non-rectifiable free algebra endomorphisms in positive characteristic, one could try to lift Gupta's construction to the free associative algebra case!

\end{document}